\documentclass[11pt,a4paper]{article}

\usepackage[headinclude]{typearea}
\usepackage[english]{babel} 
\usepackage[T1]{fontenc}
\usepackage{scrtime}
\usepackage{amsmath,amsthm,amsfonts,amssymb}
\usepackage {color}  
\usepackage {pifont}   
\usepackage {multicol,multirow}   
\usepackage{changes}
\usepackage[numbers]{natbib}  
\usepackage[normalem]{ulem}
\usepackage [scaled=.90]{helvet} 
\usepackage {courier} 
\usepackage {graphicx}
\usepackage {array} 
\usepackage {longtable}          
\usepackage{fancyvrb}
\usepackage{enumerate} 
\usepackage{ifpdf}
\usepackage{caption}
\usepackage{mathrsfs}
\usepackage{setspace}
\usepackage{marvosym}

\usepackage{mathtools}
\mathtoolsset{showonlyrefs}
\usepackage{verbatim}
\usepackage{dsfont}
\usepackage{hyperref}
\makeatletter
\def\namedlabel#1#2{\begingroup
    #2%
    \def\@currentlabel{#2}%
    \phantomsection\label{#1}\endgroup
}

\newtheorem{theorem}{Theorem}[section]

\newtheorem{lemma}[theorem]{Lemma}

\theoremstyle{definition}
\newtheorem{assumption}{Assumption}[section]

\newtheorem{remark}[theorem]{Remark}

\newcommand{\E}{{\mathbb{E}}}
\newcommand{\N}{{\mathbb{N}}}
\renewcommand{\P}{{\mathbb{P}}}

\newcommand{\R}{{\mathbb{R}}}

\newcommand{\F}{\mathcal{F}}

\newcommand{\diff}{\mathop{}\!\mathrm{d}}
\renewcommand{\P}{{\mathbb P}}

\newcommand{\bproof}{\begin{proof}}
\newcommand{\eproof}{\end{proof}}

\newcommand\independent{\protect\mathpalette{\protect\independenT}{\perp}}
\def\independenT#1#2{\mathrel{\rlap{$#1#2$}\mkern2mu{#1#2}}}

\title{Randomized Milstein algorithm for approximation of solutions of jump-diffusion SDEs}

\author{Pawe{\l} Przyby{\l}owicz \and Verena Schwarz \and Michaela Sz\"olgyenyi}
\date{Preprint, December 2023}

\begin{document}

\maketitle
%%%%%%%%%%%%%%%%%%%%%%%%%%%%%%%%%%%%%%%%%%%%%%%%%%%%%%%%%%%%%%%%%%

\begin{abstract}
We investigate the error of the randomized Milstein algorithm for solving scalar jump-diffusion stochastic differential equations. We provide a complete error analysis under substantially weaker assumptions than those known in the literature. In case the jump-commutativity condition is satisfied, we prove optimality of the randomized Milstein algorithm by establishing matching lower bounds. Moreover, we give some insight into the multidimensional case by investigating the optimal convergence rate for the approximation of jump-diffusion type L\'evys' areas. Finally, we report numerical experiments that support our theoretical findings. \newline
\end{abstract}

%%%%%%%%%%%%%%%%%%%%%%%%%%%%%%%%%%%%%%%%%%%%%%%%%%%%%%%%%%%%%%%%%%
\section{Introduction}\label{sec:intro}

Consider the following jump-diffusion stochastic differential equation (SDE)
\begin{equation}
\begin{aligned}\label{SDE}
\diff X(t) = \mu(t,X(t))\diff t +\sigma(t,X(t))\diff W(t) + \rho(t,X(t-))\diff N(t),\quad t\in[0,T], \quad X(0) = X_0, 
\end{aligned}
\end{equation}
where $\mu,\sigma,\rho:[0,T] \times \R\to\R$ are (at least) measurable functions, $T\in(0,\infty)$, $W=(W(t))_{t\in[0,T]}$ is a standard Wiener process, and $N=(N(t))_{t\in[0,T]}$ is a homogeneous Poisson process with intensity $\lambda>0$ on a filtered probability space $(\Omega, \F, (\mathcal{F}_t)_{t\geq 0},\P)$ with a filtration $(\mathcal{F}_t)_{t\geq 0}$ that satisfies the usual conditions. Furthermore, we assume $p\in [2,\infty)$ and $X_0$ to be an $\F_0$-measurable random variable with $\E[|X_0|^{2p}]<\infty$.

Due to their numerous applications in mathematical finance, control theory, and the modelling of energy markets, cf. \cite{PBL2010,PS20,PSF21,situ2005}, jump-diffusion SDEs continue to gain scientific interest. Exact solutions are only available in very special cases. It is therefore important to develop efficient (or even in some sense optimal) numerical algorithms.

The current paper serves this purpose by providing a novel numerical scheme -- the randomized Milstein scheme. 
In contrast to the classical Milstein scheme, the drift coefficient is randomized in time, that is, instead of some time grid point $t_i$ we plug a realization of $\xi_i\colon \Omega \to [t_i,t_{i+1}]$ into $\mu$. This usually improves the convergence order.
We prove upper and lower error bounds and obtain $L^2$-optimality.

Randomized algorithms are, for example, studied in \cite{PMPP2014,PMPP2017,PP2015,PP20152,Przybylowicz2021}, where the authors consider the randomized Euler--Maruyama scheme for SDEs in the jump-free case, and provide error bounds and optimality results. The articles \cite{eisenmann2018} and \cite{Heinrich2019} discuss the properties of randomized quadrature rules used for approximating stochastic It\^o integrals. 
Error bounds and optimality results of the randomized Milstein scheme for SDEs without jumps are investigated in \cite{Morkisz2020} and \cite{kruse2018}. The latter construct a two-stage version of the randomized Milstein scheme and examine its error.
Furthermore, \cite{Biswas2022} study a randomized Milstein scheme for McKean-Vlasov SDEs with common noise.
In the current paper we extend the results from \cite{Morkisz2020} and \cite{Clark1980} to provide results for jump-diffusion SDEs.

Analysis of the lower bounds and optimality is usually provided in the Information-Based Complexity (IBC) framework, see \cite{TWW88}. This setting is widely used for investigating optimal algorithms for approximation of solutions of SDEs, for example in \cite{Heinrich2019,Hertling2001,PMPP2014,PMPP2017,Morkisz2020,PP2015,PP20152,PP2016,Przybylowicz2021,MY20}.

In the current paper, we consider scalar SDEs \eqref{SDE} with coefficients that are H\"older continuous in time and Lipschitz continuous and differentiable with Lipschitz continuous derivative in space. Under these assumptions we provide upper $L^p$-error bounds for the randomized Milstein algorithm. Our assumptions are significantly weaker than any other in the literature, where it is usually assumed that the coefficients are at least twice continuously differentiable in space, cf.~\cite{Morkisz2020,PBL2010}. If additionally the so-called jump-commutativity condition (JCC) is satisfied, we prove optimality of the randomized Milstein algorithm among those randomized algorithms that use finitely many evaluations of the driving processes. It turns out that randomization of the drift coefficient in time improves the convergence rate, see Remark \ref{gain_rm} and Theorem \ref{opt_rm_scalar}. 
Our numerical results match our theoretical findings. Most interestingly our experiments suggest that for jump-diffusion SDEs the $L^p$-convergence rate depends on $p$.

As a second contribution, we study the approximation of jump-diffusion L\'evys' areas. We establish optimality of the trapezoidal rule among those algorithms that use only a finite number of evaluations of $W$ and $N$; by this extending the results from \cite{Clark1980}.
As these L\'evys' areas are naturally generated by a two-dimensional SDE, our result implies lower error bounds for any class of multidimensional SDEs that contain the generator as a subproblem.
This, in turn, implies optimality of the multidimensional Euler--Maruyama algorithm among the class of algorithms that use only finitly many evaluations of $W$ and $N$.

To summarize, the main contributions of the current paper are:
\begin{itemize}
\setlength{\itemsep}{0pt}
    \item We introduce the randomized Milstein algorithm for scalar jump-diffusion SDEs \eqref{SDE} and perform a rigorous error analysis  under relatively mild assumptions on the coefficients (Theorems \ref{ThmUpperBound} and \ref{opt_rm_scalar}). In particular, we obtain $L^2$-optimality of the randomized Milstein algorithm.
    \item We prove optimality of the approximation of jump-diffusion L\'evys' areas (Theorem \ref{lb_Levy_rea}). This yields lower error bounds for the multidimensional case whenever L\'evys' areas appear (Remark \ref{RemMultSDE}). 
    \item We perform numerical experiments that match our theoretical results and observe a $p$-dependence of the convergence rate.
\end{itemize}

The paper is organized as follows. Section 2 states the assumptions under which we perform error analysis for the randomized Milstein algorithm. Section 3 is devoted to the error analysis of the randomized Milstein process. Lower bounds and optimality analysis in the IBC framework are given in Section 4. In Section 5 we show the results of the numerical experiments. Finally, some auxiliary results used in the proofs can be found in the Appendix.

%%%%%%%%%%%%%%%%%%%%%%%%%%%%%%%%
\section{Preliminaries}\label{sec:pre}
For a random variable $X\colon\Omega\to\mathbb{R}$ we denote by $\|X\|_{L^p(\Omega)}=(\mathbb{E}[|X|]^p)^{1/p}$, where $p\in [2,\infty)$. We represent by $\lambda \in(0,\infty)$ the intensity of the Poisson process $N$. Further, we define by $\widetilde N = (\widetilde N(t))_{t\in[0,T]}$, $\widetilde N(t) = N(t)-\lambda t$ the compensated Poisson process. For $Z\in\{W,N\}$ we define by $\mathcal{F}^Z = (\mathcal{F}^Z_t)_{t\in[0,T]}$ the natural filtration with respect to $Z$.
It holds that the processes $W$ and $N$ are independent, cf.~\cite[p.~64, Theorem 97]{situ2005}, hence $\mathcal{F}_T^N \independent \mathcal{F}_T^W$.
We denote for all functions $f\in C^{0,1}([0,T]\times\R;\R)$ the partial derivative of $f$ with respect to $y$ by $\frac{\partial f}{\partial y}$. Further, we define for all functions $f\in C^{0,1}([0,T]\times\R;\R)$ the operators $L_1 f(t,y) = \sigma(t,y)\frac{\partial f}{\partial y}(t,y)$ and $L_{-1}f(t,y)=f(t,y+\rho(t,y))-f(t,y)$ for all $t\in [0,T]$, $y\in\R$.

We impose the following assumptions on the coefficient functions.

\begin{assumption}\label{Ass}
We assume for the functions $\mu,\sigma,\rho\colon[0,T]\times\R\to\R$ and for $p\in[2,\infty)$ that there exist constants $\varrho_1,\varrho_2,\varrho_3\in(0,1]$ such that:
\begin{itemize}
\item[\namedlabel{Ass1}{(i)}] For all $f\in\{\mu,\sigma,\rho\}$ holds $f\in C^{0,1}([0,T]\times\R;\R)$.
\item[\namedlabel{Ass2}{(ii)}] There exists a constant $K_1\in(0,\infty)$ such that for all $t,s\in [0,T]$, $y,z\in\R,$ and all $f\in\{\mu,\sigma,\rho\}$ it holds that
\begin{equation}
\begin{aligned}\label{AssLip}
|f(t,y)-f(t,z)|\leq K_1 |y-z|,
\end{aligned}
\end{equation}
\begin{equation}
\label{AssLipDer}
\Big|\frac{\partial f}{\partial y}(t,y) -\frac{\partial f}{\partial y}(t,z)\Big|\leq K_1 |y-z|,
\end{equation}
\begin{equation}
\label{AssHold}
|f(t,y) -f(s,y)|\leq K_1 (1+|y|)|t-s|^{\varrho_f},
\end{equation}
where $(\varrho_f,f)\in\{(\varrho_1,\mu),(\varrho_2,\sigma),(\varrho_3,\rho)\}$.
\item[\namedlabel{Ass3}{(iii)}] There exists a constant $K_2\in(0,\infty)$ such that for all $y\in\R$ and for all $t,s\in [0,T]$ holds
\begin{equation}
    \Big|\frac{\partial \mu}{\partial y}(t,y)-\frac{\partial \mu}{\partial y}(s,y)\Big|\leq K_2(1+|y|) |t-s|^{\varrho_1}.
\end{equation}
\item[\namedlabel{Ass4}{(iv)}] There exists a constant $K_3\in(0,\infty)$ such that for all $t\in [0,T]$, $y,z\in\R$, and all $f\in\{\sigma,\rho\}$ it holds that
\begin{equation}
\begin{aligned}\label{AssLipL1}
|L_1 f(t,y) - L_1 f(t,z)|\leq K_3|y-z|.
\end{aligned}
\end{equation}
\item[\namedlabel{Ass5}{(v)}] For the initial value $X_0$ we assume that it is an $\mathcal{F}_0$-measurable random variable and has finite $L^{2p}$- norm, i.e.
\begin{equation}
    \|X_0\|_{L^{2p}(\Omega)}<\infty.
\end{equation}
\end{itemize}
\end{assumption}
%%%%%%%

We obtain by the Lipschitz assumption \eqref{AssLip} that  for all $(t,y)\in [0,T]\times\R$ and $f\in\{\mu,\sigma,\rho\}$ holds
\begin{equation}
\begin{aligned}\label{LinGrowth}
|f(t,y)|\leq K_4(1+|y|),
\end{aligned}
\end{equation}
with $K_4=\max\limits_{i=1,2,3}\{\max\{|f(0,0)|,K_1\}+K_1T^{\varrho_i}\}$. Further, Assumption \ref{Ass} \ref{Ass1} and the Lipschitz continuity of $f$ in space \eqref{AssLip} imply
\begin{equation}
\begin{aligned}\label{BddDer}
\Big|\frac{\partial f}{\partial y}(t,y)\Big|\leq K_1.
\end{aligned}
\end{equation} 
Since for $f\in\{\mu,\sigma,\rho\}$ the first order partial derivative $\frac{\partial f}{\partial y}(t,\cdot)$ is Lipschitz continuous, it follows that $\frac{\partial f}{\partial y}(t,\cdot)$ is absolutely continuous. 
Hence, for all $t\in [0,T]$ the second partial derivative $\frac{\partial^2 f}{\partial y^2}(t,\cdot)$ exists almost everywhere on $\R$. Let us denote for all $t\in [0,T]$ by $S_f(t)$ the set of Lebesgue measure $0$ for which the second partial derivative $\frac{\partial^2 f}{\partial y^2}(t,\cdot)$ does not exist. Then for all $f\in\{\mu,\sigma,\rho\}$, all $t\in [0,T]$, and all $y\in\R\setminus S_f(t)$ it holds that
\begin{equation}
\begin{aligned}\label{BddSecDer}
\Big|\frac{\partial^2 f}{\partial y^2}(t,y)\Big|\leq K_1.
\end{aligned}
\end{equation}
On $S_f(t)$ we define $\frac{\partial^2 f}{\partial y^2}(t,\cdot)\equiv 0$. At this point, we like to emphasise that the choice of the values of $\frac{\partial^2 f}{\partial y^2}(t,\cdot)$ on $S_f(t)$ does not influence the proof of the main result. This is because we employ local time theory, which implies that the bounds we obtain are independent of these values. In addition, there exists a constant $K_5\in(0,\infty)$ such that for all $f\in\{\sigma,\rho\}$ and for all $(t,y)\in [0,T]\times\R$ holds
\begin{equation}
\label{ling_l1l1}
    \max\{|L_1f(t,y)|,|L_{-1}f(t,y)|\}\leq K_5(1+|y|).
\end{equation}

The existence and uniqueness of a strong solution to the SDE \eqref{SDE} is well-known under Assumption \ref{Ass}, e.g.~\cite[p.~255, Theorem 6]{Protter2005}. Further, since  $\mathbb{E}[|X_0|^{2p}]<\infty$ by \cite[Lemma 1]{Przybylowicz2021} there exists $K_6\in(0,\infty)$ such that it holds
\begin{equation}
\begin{aligned}
\label{MomEstSol}
\E\Big[\sup_{0\leq t\leq T} |X(t)|^{2p}\Big]\leq K_6,
\end{aligned}
\end{equation}
and further for all $s,t\in[0,T]$ holds 
\begin{equation}
\begin{aligned}
\label{MomEstIncSol}
\E\Big[|X(t)-X(s)|^p\Big]\leq K_6 |t-s|.
\end{aligned}
\end{equation}
Note that the estimate \eqref{MomEstIncSol} can be improved if $\rho \equiv 0$.

Under the Assumption \ref{Ass} (i) and (ii) the Meyer-It\^o formula \cite[p.~221, Theorem  71]{Protter2005} is applicable to the function $\R\ni y\mapsto f(t,y)\in \mathbb{R}$ and the solution process $(X(s))_{s\in [v_1,v_2]}$ for all $f\in\{\mu,\sigma,\rho\}$ and all $t\in [v_1,v_2]\subset [0,T]$. We obtain the following parametric version of the Meyer-It\^o formula: For all $s,t\in [v_1,v_2]$ it holds that
\begin{equation}
\label{MeyerItoFormula}
f(t,X(s)) = f(t,X(v_1)) + \int\limits_{v_1}^s \alpha(f,t,u)\diff u + \int\limits_{v_1}^s \beta(f,t,u)\diff W(u)
+\int\limits_{v_1}^s \gamma(f,t,u)\diff N(u),
\end{equation}
where 
\begin{equation}
\begin{aligned}\label{LFormulas}
\alpha(f,t,u) &= \alpha_1(f,t,u) + \alpha_2(f,t,u),\\
\alpha_1(f,t,u) &= \frac{\partial f}{\partial y}(t,X(u))\mu(u,X(u)),\\
\alpha_2(f,t,u) &= \frac{1}{2} \frac{\partial^2 f}{\partial y^2}(t,X(u))\sigma^2(u,X(u)),\\
\beta(f,t,u) &= \frac{\partial f}{\partial y}(t,X(u))\sigma(u,X(u)),\\
\gamma(f,t,u) &= f(t,X(u-)+\rho(u,X(u-)))-f(t,X(u-)).
\end{aligned}
\end{equation}
We refer to Lemma \ref{MEstL} for fundamental estimations of the above defined functions. 

Next, we define the randomized Milstein algorithm. For $n\in\N$ we set $\delta =T/n$ and define $ t_i = i\delta$ for all $ i\in\{0,\ldots,n\}$. Further, we use the notation $
\Delta Y_i = Y(t_{i+1}) - Y(t_i)$  for all $i\in\{0,1,\ldots,n-1\}$
and 
\[\displaystyle{I_{s,t}(Y,Z) = \int\limits_s^t\int\limits_s^{u-} \diff Y(v) \diff Z(u)}\]
for all $Y,Z\in\{W,N\}$ and $s,t\in[0,T]$. It holds that
\begin{equation}
\label{JCC_int}
    I_{s,t}(N,W)+I_{s,t}(W,N)=(W(t)-W(s))(N(t)-N(s))
\end{equation}
and that the sigma-field generated by $I_{s,t}(Y,Z)$ and $\F_s$ are independent, \cite[Fact B.28 (ii)]{AK2021}.  Let $\{\xi_i\}_{i=0}^{n-1}$ be independent random variables on the probability space $(\Omega,\F,\mathbb{P})$, such that the sigma-field generated by $\xi_0,\xi_1,\ldots,\xi_{n-1}$ and $\F_{T}$ are independent and $\xi_i$ is uniformly distributed on $[t_i,t_{i+1}]$ for $i\in\{0,\ldots, n-1\}$. Then the randomized Milstein algorithm $X^{(\delta)}$ is defined recursively through
\begin{equation}
\begin{aligned}\label{MS}
X^{(\delta)}(t_0) &= X_0,\\
X^{(\delta)}(t_{i+1}) &= X^{(\delta)}(t_{i}) + \mu(\xi_i,X^{(\delta)}(t_{i})) \delta + \sigma(t_{i},X^{(\delta)}(t_{i}))\Delta W_i +\rho(t_{i},X^{(\delta)}(t_{i})) \Delta N_i\\
&\quad+ L_1\sigma(t_{i},X^{(\delta)}(t_{i})) I_{t_i,t_{i+1}}(W,W)
+ L_{-1}\rho(t_{i},X^{(\delta)}(t_{i})) I_{t_i,t_{i+1}}(N,N)\\
&\quad+ L_{-1}\sigma(t_{i},X^{(\delta)}(t_{i})) I_{t_i,t_{i+1}}(N,W)
+ L_{1}\rho(t_{i},X^{(\delta)}(t_{i})) I_{t_i,t_{i+1}}(W,N),\\
&\quad\quad\quad\quad\quad\quad\quad\quad \quad\quad\quad\quad\quad\quad\quad\quad \quad\quad\quad\quad\quad\quad\quad\quad i \in\{ 0,\ldots, n-1\}.
\end{aligned}
\end{equation}
To analyse the error of the randomized Milstein algorithm we additionally define the time-continuous Milstein approximation $(X^{(\delta)}_c(t))_{t\in[0,T]}$. It is also known as the randomized Milstein process and is defined as
\begin{equation}
\begin{aligned}\label{MP}
X^{(\delta)}_c(t_0) &= X_0,\\
X^{(\delta)}_c(t) &= X^{(\delta)}_c(t_{i}) + \mu(\xi_i,X^{(\delta)}_c(t_{i})) (t-t_{i}) + \sigma(t_{i},X^{(\delta)}_c(t_{i}))(W(t) - W(t_i))\\ 
&\quad +\rho(t_{i},X^{(\delta)}_c(t_{i}))( N(t) - N(t_i))\\
&\quad+ L_1\sigma(t_{i},X^{(\delta)}_c(t_{i})) I_{t_i,t}(W,W)
+ L_{-1}\rho(t_{i},X^{(\delta)}_c(t_{i})) I_{t_i,t}(N,N)\\
&\quad+ L_{-1}\sigma(t_{i},X^{(\delta)}_c(t_{i})) I_{t_i,t}(N,W)
+ L_{1}\rho(t_{i},X^{(\delta)}_c(t_{i})) I_{t_i,t}(W,N),
\end{aligned}
\end{equation}
for $t\in(t_i,t_{i+1}]$, $i \in\{0,\ldots, n-1\}$. Then it holds that for all $i\in\{0,\ldots,n\}$ that $X^{(\delta)}(t_{i}) = X^{(\delta)}_c(t_{i})$. 

Now, similarly to \cite{Morkisz2020}, we expand the filtration $(\mathcal{F}_{t})_{t\in[0,T]}$ in the following way: We denote the sigma-algebra generated by $\xi_0,\ldots,\xi_{n-1}$ as $\mathcal{G}^n$, and $\mathcal{\bar F}^n_t$ as the sigma-algebra generated by $\mathcal{F}_t$ and $\mathcal{G}^n$. Since $\mathcal{G}^n$ and $\mathcal{F}_{T}$ are independent, $W$ and $N$ are still Wiener and Poisson processes with respect to $(\mathcal{\bar F}^n_{t})_{t\in[0,T]}$, respectively. Seeing that in this article we are integrating 
\begin{itemize}
\setlength{\itemsep}{0pt}
    \item $(\mathcal{\bar F}^n_{t})_{t\in[0,T]}$-progressively measurable processes with respect to the continuous $(\mathcal{\bar F}^n_{t})_{t\in[0,T]}$-semimartingales $(t)_{t\in [0,T]}$, $(W(t))_{t\in [0,T]}$,
    \item $(\mathcal{\bar F}^n_{t})_{t\in[0,T]}$-adapted càglàd processes with respect to the càdlàg $(\mathcal{\bar F}^n_{t})_{t\in[0,T]}$-semimartingale $(N(t))_{t\in [0,T]}$,
\end{itemize}  
the (stochastic) integrals are well-defined, e.g.~\cite{Protter2005}. Moreover, the randomized Milstein process is  $(\mathcal{\bar F}^n_{t})_{t\in[0,T]}$-progressively measurable, since it is càdlàg and adapted.

Note that the randomised Milstein process cannot be implemented because it requires all values of $W$ and $N$, which are inaccessible. However, we will use it as an auxiliary scheme for our proof that the randomized Milstein algorithm has convergence order $\delta^{\min\{\frac{2}{p},\varrho_1+\frac{1}{p},\varrho_2,\varrho_3\}}$. 
%%%%%%%%%%%%%%%%%%%
\section{Error analysis for the randomized Milstein process}
Let for all $i\in\{1,\ldots,n\}$,
\begin{equation}
    U_i=(t_i,X^{(\delta)}(t_{i})), \quad V_i=(\xi_i,X^{(\delta)}(t_{i})).
\end{equation}
The processes $X$ and $X^{(\delta)}_c$ can be written for all  $t\in[0,T]$ as
\begin{equation}
\begin{aligned}\label{XNot}
X(t) =X(0) + A(t) + B(t) + C(t),
\end{aligned}
\end{equation}
\begin{equation}
\begin{aligned}
\label{XdeltaNot}
X^{(\delta)}_c(t) = X(0) + A^{(\delta)}(t) + B^{(\delta)}(t) + C^{(\delta)}(t),
\end{aligned}
\end{equation}
where 
\begin{equation}
\begin{aligned}
A(t) &= \int\limits_0^t \sum_{i=0}^{n-1} \mu(s,X(s))\mathds{1}_{(t_i,t_{i+1}]}(s) \diff s,\\
B(t) &= \int\limits_0^t \sum_{i=0}^{n-1} \sigma(s,X(s))\mathds{1}_{(t_i,t_{i+1}]}(s) \diff W(s),\\
C(t) &= \int\limits_0^t \sum_{i=0}^{n-1} \rho(s,X(s-))\mathds{1}_{(t_i,t_{i+1}]}(s) \diff N(s),
\end{aligned}
\end{equation}
\begin{equation}
\begin{aligned}
A^{(\delta)}(t) &= \int\limits_0^t \sum_{i=0}^{n-1} \mu(V_i)\mathds{1}_{(t_i,t_{i+1}]}(s) \diff s,\\
B^{(\delta)}(t) &= \int\limits_0^t \sum_{i=0}^{n-1} \Big( \sigma(U_i) + \int\limits_{t_i}^s L_1 \sigma(U_i)\diff W(u) + \int\limits_{t_i}^s L_{-1} \sigma(U_i)\diff N(u)\Big)\mathds{1}_{(t_i,t_{i+1}]}(s) \diff W(s),\\
C^{(\delta)}(t) &= \int\limits_0^t \sum_{i=0}^{n-1} \Big( \rho(U_i) + \int\limits_{t_i}^s L_1 \rho(U_i)\diff W(u) + \int\limits_{t_i}^{s-} L_{-1} \rho(U_i)\diff N(u)\Big)\mathds{1}_{(t_i,t_{i+1}]}(s) \diff N(s).
\end{aligned}
\end{equation}
%%%%%%%%
\begin{lemma}
\label{Milstein_proc_lemma}
Under the Assumption \ref{Ass} there exists a constant $K_8 \in (0,\infty)$ such that for all $n\in\N$ it holds that
\begin{equation}
\begin{aligned}\label{MomEst}
\sup_{0\leq t\leq T} \E\big[|X^{(\delta)}_c(t)|^p\big] \leq K_8.
\end{aligned}
\end{equation}
\end{lemma}
\begin{proof}
Knowing that $\mathbb{E}[|X_0|^p]<\infty$ we obtain by induction that
\begin{equation}
\label{est_mom_ind}
    \max\limits_{0\leq i\leq n}\E\big[|X^{(\delta)}(t_i)|^p\big]<\infty.
\end{equation}
Further, by \eqref{est_mom_ind} and \eqref{MP} for all $n\in\N$ exists a constant $c_1\in(0,\infty)$ such that
\begin{equation}
\label{finite_mom_lp}
    \sup\limits_{0\leq t\leq T}\mathbb{E}\big[|X^{\delta}_c(t)|^p\big]\leq c_1(1+\max\limits_{0\leq i\leq n-1}\mathbb{E}\big[|X^{(\delta)}(t_i)|^p\big])<\infty.
\end{equation}

We denote for all $t\in [0,T]$,
\begin{equation}\label{41eq1}
    X_c^{(\delta)}(t)=X(0)+\int\limits_0^t \Psi_{1,n}(s)\diff s+\int\limits_0^t \Psi_{2,n}(s)\diff W(s)+\int\limits_0^t \Psi_{3,n}(s)\diff N(s),
\end{equation}
where
\begin{equation}\label{41eq2}
    \Psi_{1,n}(s)=\sum_{i=0}^{n-1} \mu(V_i)\mathds{1}_{(t_i,t_{i+1}]}(s),
\end{equation}
\begin{equation}\label{41eq3}
    \Psi_{2,n}(s)= \sum_{i=0}^{n-1} \Bigg( \sigma(U_i) + \int\limits_{t_i}^s L_1 \sigma(U_i)\diff W(u)+ \int\limits_{t_i}^s L_{-1} \sigma(U_i)\diff N(u)\Bigg)\mathds{1}_{(t_i,t_{i+1}]}(s),
\end{equation}
\begin{equation}\label{41eq4}
    \Psi_{3,n}(s)= \sum_{i=0}^{n-1} \Bigg( \rho(U_i) + \int\limits_{t_i}^s L_1 \rho(U_i)\diff W(u)+ \int\limits_{t_i}^{s-} L_{-1} \rho(U_i)\diff N(u)\Bigg)\mathds{1}_{(t_i,t_{i+1}]}(s).
\end{equation}
By Lemma \ref{BDGaKunita} holds for all $(k,Z)\in \{(1,s), (2,W), (3,N)\}$ that
\begin{equation}\label{41eq5}
    \mathbb{E}\Bigg[\Biggl|\int\limits_0^t \Psi_{k,n}(s)dZ(s)\Biggl|^p\Bigg]\leq \hat c\int\limits_0^t\mathbb{E}\big[|\Psi_{k,n}(s)|^p\big] ds.
\end{equation}
By \eqref{LinGrowth} there exist constants $c_2,c_3\in(0,\infty)$ such that
\begin{equation}\label{41eq6}
    \mathbb{E}\big[|\Psi_{1,n}(s)|^p\big]
    \leq K_4^p\sum\limits_{i=0}^{n-1}\mathbb{E}\big[(1+|X^{(\delta)}(t_i)|)^p\big]\cdot \mathds{1}_{(t_i,t_{i+1}]}(s)
    \leq c_2+c_3\sum\limits_{i=0}^{n-1}\mathbb{E}\big[|X^{(\delta)}(t_i)|^p\big]\cdot \mathds{1}_{(t_i,t_{i+1}]}(s).
\end{equation}
Using \eqref{LinGrowth}, \eqref{ling_l1l1}, and Lemma \ref{BDGaKunita} we obtain that there exist constants $c_4,c_5\in(0,\infty)$ such that for all $(k,f)\in\{(2,\sigma), (3,\rho)\}$ holds
\begin{equation}
\begin{aligned}\label{41eq7}
    \int\limits_0^t \mathbb{E}\big[|\Psi_{k,n}(s)|^p\big] \diff s&\leq  \hat c\, \E\Bigg[ \sum_{i=0}^{n-1}\int\limits_0^t | f(U_i) |^p \mathds{1}_{(t_i,t_{i+1}]}(s) \diff s \Bigg]\\
    &\quad  +  \hat c\, \E\Bigg[ \sum_{i=0}^{n-1}\int\limits_0^t \Bigg| \int\limits_{t_i}^s L_1 f(U_i)\diff W(u)\Bigg|^p \mathds{1}_{(t_i,t_{i+1}]}(s) \diff s \Bigg]\\ 
    &\quad +  \hat c\, \E\Bigg[ \sum_{i=0}^{n-1}\int\limits_0^t \Bigg| \int\limits_{t_i}^s L_{-1} f(U_i)\diff N(u)\Bigg|^p \mathds{1}_{(t_i,t_{i+1}]}(s) \diff s \Bigg]\\
    &    \leq c_4+c_5\int\limits_0^t \sum\limits_{i=0}^{n-1}\mathbb{E}\big[|X^{(\delta)}(t_i)|^p\big]\cdot \mathds{1}_{(t_i,t_{i+1}]}(s)\diff s.
    \end{aligned}
\end{equation}
In this step we used that $\displaystyle{\int\limits_{t_i}^s L_{-1} f(U_i)\diff N(u)}$ and $\displaystyle{\int\limits_{t_i}^{s-} L_{-1} f(U_i)\diff N(u)}$ differ at most in finitely many points. Combining \eqref{41eq1}, \eqref{41eq6}, and \eqref{41eq7} we obtain that there exist constants $c_6,c_7,c_8\in(0,\infty)$ such that
\begin{equation}
    \begin{aligned}\label{41eq8}
    &\E\Big[ \big|X_c^{(\delta)}(t)\big|^p\Big]\leq c_6\Bigg( \E\Big[|X(0)|^p\Big] + \sum_{k=1}^{3} \int\limits_0^t \E\Big[ |\Psi_{k,n}(s) |^p\Big] \diff s \Bigg)\\
    &\leq c_7 \Big(\E\Big[|X(0)|^p\Big]+1 \Big) + c_8 \int\limits_0^t \sup_{0\leq u\leq s} \E\Big[\big|X_c^{(\delta)}(u)\big|^p\Big]\diff s.
    \end{aligned}
\end{equation}
Hence,
\begin{equation}
\begin{aligned}\label{41eq81}
    &\sup_{0\leq s \leq t}\E\Big[ \big|X_c^{(\delta)}(s)\big|^p\Big]
    \leq c_7 \Big(\E\Big[|X(0)|^p\Big]+1 \Big) + c_8 \int\limits_0^t \sup_{0\leq u\leq s} \E\Big[\big|X_c^{(\delta)}(u)\big|^p\Big]\diff s.
    \end{aligned}
\end{equation}
The mapping $t\mapsto \sup_{0\leq s \leq t}\E\Big[ \big|X_c^{(\delta)}(s)\big|^p\Big]$ is Borel measurable since it is monotone. In addition it is bounded by \eqref{finite_mom_lp}. Consequently, applying Gronwall's lemma proves the claim. 
\end{proof}

Next we prove upper bounds for the randomized Milstein algorithm's convergence rate.

\begin{theorem}
\label{ThmUpperBound}
Let Assumption \ref{Ass} hold. Then there exists  $C\in (0,\infty)$ such that for all $n\in\N$ holds
\begin{equation}
\begin{aligned}\label{Main}
\sup_{0\leq t\leq T} \|X(t) - X^{(\delta)}_c(t)\|_{L^p(\Omega)}
\leq  C\delta^{\min\{\frac{2}{p},\varrho_1+\frac{1}{p},\varrho_2,\varrho_3\}}.
\end{aligned}
\end{equation}
\end{theorem}

\begin{proof}
For all $t\in[0,T]$ it holds that 
\begin{equation}
    \begin{aligned}\label{Meq1}
    X(t)- X^{(\delta)}_c(t) = \big(A(t) -A^{(\delta)}(t) \big)+ \big(B(t) -B^{(\delta)}(t) \big) + \big(C(t) -C^{(\delta)}(t) \big).
    \end{aligned}
\end{equation}
We begin by rewriting each summand of the right hand side of equation \eqref{Meq1}. We obtain
\begin{equation}
    \begin{aligned}\label{Meq3}
    A(t) -A^{(\delta)}(t)
    &=\tilde A^{(\delta)}_1(t)+\tilde A^{(\delta)}_2(t)+\tilde A^{(\delta)}_3(t),
    \end{aligned}
\end{equation}
where 
\begin{equation}
    \begin{aligned}\label{Meq3a}
    \tilde A^{(\delta)}_1(t)&=\int\limits_0^t \sum_{i=0}^{n-1} \big( \mu(s,X(s))- \mu(s,X(t_{i})) \big) \mathds{1}_{(t_i,t_{i+1}]}(s) \diff s,\\
    \tilde A^{(\delta)}_2(t)&=\int\limits_0^t \sum_{i=0}^{n-1} \big( \mu(s,X(t_i))- \mu(\xi_i,X(t_{i})) \big) \mathds{1}_{(t_i,t_{i+1}]}(s) \diff s,\\
    \tilde A^{(\delta)}_3(t)&= \int\limits_0^t \sum_{i=0}^{n-1}\big(\mu(\xi_i,X(t_{i}))- \mu(\xi_i,X^{(\delta)}_c(t_{i})) \big) \mathds{1}_{(t_i,t_{i+1}]}(s) \diff s.
    \end{aligned}
\end{equation}
We apply the parametric version of the Meyer-It\^o formula \eqref{MeyerItoFormula} to obtain
\begin{equation}
    \mu(s,X(s))-\mu(s,X(t_i))=\int\limits_{t_i}^s \alpha(\mu,s,u) \diff u +\int\limits_{t_i}^s \beta(\mu,s,u)\diff 
    W(u) +\int\limits_{t_i}^s\gamma(\mu,s,u)\diff N(u).
\end{equation}
Hence,
\begin{equation}\label{Meq4}
    \tilde A_1^{(\delta)}(t) = \sum\limits_{j=1}^3 \tilde M_j^{(\delta)}(t),
\end{equation}
where
\begin{equation}
    \begin{aligned}
    &\tilde M_1^{(\delta)}(t)=\int\limits_0^t\sum\limits_{i=0}^{n-1}\Bigg(\int\limits_{t_i}^s\alpha(\mu,s,u)\diff u\Bigg)\cdot\mathds{1}_{(t_i,t_{i+1}]}(s)\diff s,\\
    &\tilde M_2^{(\delta)}(t)=\int\limits_0^t\sum\limits_{i=0}^{n-1}\Bigg(\int\limits_{t_i}^s\beta(\mu,s,u)\diff W(u)\Bigg)\cdot\mathds{1}_{(t_i,t_{i+1}]}(s)\diff s,\\
    &\tilde M_3^{(\delta)}(t)=\int\limits_0^t\sum\limits_{i=0}^{n-1}\Biggl(\int\limits_{t_i}^s\gamma(\mu,s,u)\diff N(u)\Biggr)\cdot\mathds{1}_{(t_i,t_{i+1}]}(s)\diff s.
    \end{aligned}
\end{equation}
Furher, we obtain for the second summand of \eqref{Meq1},
\begin{equation}
    \begin{aligned}\label{Meq5}
    B(t) -B^{(\delta)}(t)
    &= \int\limits_0^t \sum_{i=0}^{n-1} \big(\sigma(s,X(s)) - \sigma(t_i,X(s))\big) \mathds{1}_{(t_i,t_{i+1}]}(s) \diff W(s) \\
    &\quad + \int\limits_0^t \sum_{i=0}^{n-1} \Bigg(\sigma(t_i,X(s))- \sigma(t_i,X(t_i))\\
    &\quad\quad\quad\quad\quad\quad-\int\limits_{t_i}^s L_1 \sigma(U_i)\diff W(u)
    - \int\limits_{t_i}^s L_{-1} \sigma(U_i)\diff N(u)\Bigg)\mathds{1}_{(t_i,t_{i+1}]}(s) \diff W(s)\\
    &\quad + \int\limits_0^t \sum_{i=0}^{n-1} \big(\sigma(t_i,X(t_i)) - \sigma(U_i)\big) \mathds{1}_{(t_i,t_{i+1}]}(s) \diff W(s).
    \end{aligned}
\end{equation}
Again we obtain by using the parametric version of the Meyer-It\^o formula  \eqref{MeyerItoFormula} that
\begin{equation}
    \begin{aligned}
    &\sigma(t_i,X(s))- \sigma(t_i,X(t_i))-\int\limits_{t_i}^s L_1 \sigma(U_i)\diff W(u) 
    - \int\limits_{t_i}^s L_{-1} \sigma(U_i)\diff N(u)\\
    &=\int\limits_{t_i}^s \alpha(\sigma,t_i,u)\diff u+\int\limits_{t_i}^s \Bigl(\beta(\sigma,t_i,u)-L_1 \sigma(U_i)\Bigr)\diff W(u)
    +\int\limits_{t_i}^s \Bigl(\gamma(\sigma,t_i,u)-L_{-1} \sigma(U_i)\Bigr)\diff N(u).
    \end{aligned}
\end{equation}
Hence, 
\begin{equation}
    \begin{aligned}\label{Meq6}
    B(t) -B^{(\delta)}(t)&= \int\limits_0^t \sum_{i=0}^{n-1} \Bigg( \int\limits\limits_{t_i}^s \alpha(\sigma,t_i,u)\diff u\Bigg) \mathds{1}_{(t_i,t_{i+1}]}(s) \diff W(s)\\
    &\quad + \int\limits_0^t \sum_{i=0}^{n-1}\Biggl( \int\limits\limits_{t_i}^s \Bigl(\beta(\sigma,t_i,u)-L_1 \sigma(U_i)\Bigr)\diff W(u)\Biggr) \mathds{1}_{(t_i,t_{i+1}]}(s) \diff W(s)\\
    &\quad + \int\limits_0^t \sum_{i=0}^{n-1} \Biggl(\int\limits_{t_i}^s \Bigl(\gamma(\sigma,t_i,u)-L_{-1} \sigma(U_i)\Bigr)\diff N(u) \Biggr) \mathds{1}_{(t_i,t_{i+1}]}(s) \diff W(s) \\
    &\quad +\int\limits_0^t \sum_{i=0}^{n-1} \big(\sigma(s,X(s)) - \sigma(t_i,X(s))\big) \mathds{1}_{(t_i,t_{i+1}]}(s) \diff W(s) \\
    &\quad + \int\limits_0^t \sum_{i=0}^{n-1} \big(\sigma(t_i,X(t_i)) - \sigma(U_i)\big) \mathds{1}_{(t_i,t_{i+1}]}(s) \diff W(s).
    \end{aligned}
\end{equation}
For the third summand of \eqref{Meq1} it holds that
\begin{equation}
    \begin{aligned}\label{Meq7}
    C(t) -C^{(\delta)}(t)
    &= \int\limits_0^t \sum_{i=0}^{n-1} \big(\rho(s,X(s-)) - \rho(t_i,X(s-))\big) \mathds{1}_{(t_i,t_{i+1}]}(s) \diff N(s) \\
    &\quad + \int\limits_0^t \sum_{i=0}^{n-1} \Bigg(\rho(t_i,X(s-))- \rho(t_i,X(t_i))\\
    &\quad\quad\quad\quad\quad\quad-\int\limits_{t_i}^{s} L_1 \rho(U_i)\diff W(u)
    - \int\limits_{t_i}^{s-} L_{-1} \sigma(U_i)\diff N(u)\Bigg)\mathds{1}_{(t_i,t_{i+1}]}(s) \diff N(s)\\
    &\quad + \int\limits_0^t \sum_{i=0}^{n-1} \big(\rho(t_i,X(t_i)) - \rho(U_i)\big) \mathds{1}_{(t_i,t_{i+1}]}(s) \diff N(s).
    \end{aligned}
\end{equation}
Using the parametric version of the Meyer-It\^o formula \eqref{MeyerItoFormula} we obtain
\begin{equation}
    \begin{aligned}
    &\rho(t_i,X(s-))- \rho(t_i,X(t_i))-\int\limits_{t_i}^s L_1 \rho(U_i)\diff W(u)
    - \int\limits_{t_i}^{s-} L_{-1} \rho(U_i)\diff N(u)\\
    &=\int\limits_{t_i}^s \alpha(\rho,t_i,u)\diff u+\int\limits_{t_i}^s \Bigl(\beta(\rho,t_i,u)-L_1 \rho(U_i)\Bigr)\diff W(u)
    +\int\limits_{t_i}^{s-} \Bigl(\gamma(\rho,t_i,u)-L_{-1} \rho(U_i)\Bigr)\diff N(u).
    \end{aligned}
\end{equation}
Due to the continuity of the processes it holds that
\begin{equation}
   \int\limits_{t_i}^s \alpha(\rho,t_i,u)\diff u=\int\limits_{t_i}^{s-} \alpha(\rho,t_i,u)\diff u, \quad \int\limits_{t_i}^s\beta(\rho,t_i,u)\diff W(u)=\int\limits_{t_i}^{s-}\beta(\rho,t_i,u)\diff W(u).
\end{equation}
Consequently, for all $t\in [0,T]$ holds
\begin{equation}
    \begin{aligned}\label{Meq8}
    C(t) -C^{(\delta)}(t)&= \int\limits_0^t \sum_{i=0}^{n-1} \Bigg( \int\limits_{t_i}^s \alpha(\rho,t_i,u)\diff u\Bigg) \mathds{1}_{(t_i,t_{i+1}]}(s) \diff N(s)\\
    &\quad + \int\limits_0^t \sum_{i=0}^{n-1}\Biggl( \int\limits_{t_i}^s \Bigl(\beta(\rho,t_i,u)-L_1 \rho(U_i)\Bigr)\diff W(u)\Biggr) \mathds{1}_{(t_i,t_{i+1}]}(s) \diff N(s)\\
    &\quad + \int\limits_0^t \sum_{i=0}^{n-1} \Biggl(\int\limits_{t_i}^{s-} \Bigl(\gamma(\rho,t_i,u)-L_{-1} \rho(U_i)\Bigr)\diff N(u) \Biggr) \mathds{1}_{(t_i,t_{i+1}]}(s) \diff N(s) \\
    &\quad +\int\limits_0^t \sum_{i=0}^{n-1} \big(\rho(s,X(s-)) - \rho(t_i,X(s-))\big) \mathds{1}_{(t_i,t_{i+1}]}(s) \diff N(s) \\
    &\quad + \int\limits_0^t \sum_{i=0}^{n-1} \big(\rho(t_i,X(t_i)) - \rho(U_i)\big) \mathds{1}_{(t_i,t_{i+1}]}(s) \diff N(s). 
    \end{aligned}
\end{equation}

We now estimate all terms in \eqref{Meq3}, \eqref{Meq4}, \eqref{Meq6}, and \eqref{Meq8}.  
We use Lemma \ref{BDGaKunita} and Assumption \ref{Ass} \ref{Ass1} for $(f,v,Z)\in\{(\mu,\xi_i,\operatorname{Id}),(\sigma,t_i, W), (\rho, t_i, N)\}$. This implies the existence of a constant $c_1\in(0,\infty)$ such that
\begin{equation}
    \begin{aligned}\label{Meq10}
    &\E\Bigg[\Bigg|\int\limits_0^t \sum_{i=0}^{n-1} \big(f(v,X(t_i))- f(v,X^{(\delta)}_c(t_i)) \big) \mathds{1}_{(t_i, t_{i+1}]}(s) \diff Z(s) \Bigg|^p\Bigg]\\
    &\leq \hat c\, \E\Bigg[\int\limits_0^t \sum_{i=0}^{n-1} \big| f(v,X(t_i))- f(v,X^{(\delta)}_c(t_i)) \big|^p \mathds{1}_{(t_i, t_{i+1}]}(s) \diff s\Bigg]\\
    &\leq c_1 \int\limits_0^t \sum_{i=0}^{n-1} \E\Big[\big| X(t_i)- X^{(\delta)}_c(t_i) \big|^p\Big] \mathds{1}_{(t_i, t_{i+1}]}(s) \diff s.
    \end{aligned}
\end{equation}
Further, using Lemma \ref{BDGaKunita} and \eqref{AssHold} we obtain the existence of a constant $c_2\in(0,\infty)$ such that for all $(f,Z)\in\{(\sigma,W),(\rho,N)\}$ and for all $t\in [0,T]$ holds 
\begin{equation}
    \begin{aligned}\label{Meq10a}
    &\mathbb{E}\Bigg[\Bigg|\int\limits_0^t\sum\limits_{i=0}^{n-1}\Bigl(f(s,X(s-))-f(t_i,X(s-))\Bigr)\mathds{1}_{(t_i,t_{i+1}]}(s)\diff Z(s)\Bigg|^p\Bigg]\\
    &\leq \hat c \sum\limits_{i=0}^{n-1}\mathbb{E}\Bigg[\int\limits_{t_i}^{t_{i+1}}|f(s,X(s-))-f(t_i,X(s-))|^p\diff s\Bigg]
    \\
    &\leq \hat c K_1^p \sum\limits_{i=0}^{n-1}\mathbb{E}\Bigg[\int\limits_{t_i}^{t_{i+1}} (1+|X(s-)|)^p\cdot (s-t_i)^{p\varrho_f}\diff s\Bigg]\\
    &\leq \hat c K_1^p \delta^{p\varrho_f}\sum\limits_{i=0}^{n-1}\mathbb{E}\Bigg[\int\limits_{t_i}^{t_{i+1}} (1+|X(s)|)^p\diff s\Bigg]\leq 2^{p-1}\hat cK_1^p \delta^{p\varrho_f}\Bigl(1+\mathbb{E}\Big[\sup\limits_{0\leq t\leq T}|X(t)|^p\Big]\Bigr)\leq c_2\delta^{p\varrho_f}.
    \end{aligned}
\end{equation}
Using Lemma \ref{BDGaKunita} we get that there exist constants $c_3,c_4\in(0,\infty)$ such that for $(f,v,Z)\in\{(\mu,s,\operatorname{Id}),(\sigma,t_i, W), (\rho,t_i, N)\}$ holds
\begin{equation}
\begin{aligned}
\label{Meq11}
&\E\Bigg[\Bigg|\int\limits_0^t \sum_{i=0}^{n-1} \Bigg( \int\limits_{t_i}^s \alpha(f,v,u) \diff u \Bigg) \mathds{1}_{(t_i, t_{i+1}]}(s) \diff Z(s) \Bigg|^p\Bigg]\\
&\leq c_3\,\sum_{i=0}^{n-1}\int\limits_{t_i}^{t_{i+1}}\E\Bigg[ \Bigg(\int\limits_{t_i}^{s} |\alpha_1(f,v,u)| \diff u \Bigg)^p\Bigg]\diff s+c_4\, \sum_{i=0}^{n-1}\int\limits_{t_i}^{t_{i+1}}\E\Bigg[ \Bigg(\int\limits_{t_i}^{s} |\alpha_2(f,v,u)| \diff u \Bigg)^p\Bigg]\diff s.
\end{aligned}
\end{equation}
The expectations in equation \eqref{Meq11} are then estimated one by one. 
For the first term, we utilise \eqref{LinGrowth}, \eqref{BddDer}, and \eqref{MomEstSol} to obtain the existence of a constant $c_5\in(0,\infty)$ such that for all $s\in [t_i,t_{i+1}]$, $v\in\{s,t_i\}$ holds
\begin{equation}
\begin{aligned}\label{Meq12}
&\E\Bigg[ \Bigg(\int\limits_{t_i}^{s}  |\alpha_1(f,v,u)| \diff u \Bigg)^p\Bigg]
= \E\Bigg[\Bigg(\int\limits_{t_i}^{s} \Big|\frac{\partial f}{\partial y}(v,X(u))\Big|\cdot \big| \mu(u,X(u))\big| \diff u \Bigg)^p\Bigg]\\
&\leq (K_1 K_4)^p\, \E\Bigg[\Bigg(\int\limits_{t_i}^{s} \big( 1 + |X(u)|\big) \diff u \Bigg)^p\Bigg]
\leq (K_1 K_4)^p\delta^p\, \E\Big[\big( 1 + \sup_{0\leq t \leq T} |X(t)|\big)^p\Big]
\leq c_5\, \delta^p.
\end{aligned}
\end{equation}
For the second term we use \eqref{LinGrowth}, the fact that $|\frac{\partial^2 f}{\partial y^2}(t,y)| \leq K_1$ for all $t\in[0,T]$ and $y\in\R$ by \eqref{BddSecDer}, and \eqref{MomEstSol} to obtain that there exist constants $c_6,c_7,c_{8}\in(0,\infty)$ such that for all $s\in [t_i,t_{i+1}]$, $v\in \{s,t_i\}$,
\begin{equation}
\begin{aligned}\label{Meq13}
&\E\Bigg[ \Bigg(\int\limits_{t_i}^{s} |\alpha_2(f,v,u)| \diff u \Bigg)^p\Bigg]
= \E\Bigg[\Bigg(\frac{1}{2} \int\limits_{t_i}^{s} \Big|\frac{\partial^2 f}{\partial y^2}(v,X(u))\Big|\cdot \big| \sigma^2(u,X(u))\big| \diff u \Bigg)^p\Bigg]\\
&\leq \Big(\frac{K_1}{2}\Big)^p \E\Bigg[\Bigg( \int\limits_{t_i}^{s} \sigma^2(u,X(u)) \diff u \Bigg)^p\Bigg]
\leq c_6\, \E\Bigg[\Bigg( \int\limits_{t_i}^{t_{i+1}} \big( 1 + |X(u)|^2\big) \diff u \Bigg)^p\Bigg]\\
&\leq c_7\, \delta^p\, \E\Big[ \big( 1 + \sup_{0\leq t\leq T}|X(t)|^2\big)^p\Big]
\leq c_{8}\, \delta^p.
\end{aligned}
\end{equation}
As a result of combining equations \eqref{Meq11}, \eqref{Meq12}, and \eqref{Meq13}, we conclude that there exists a constant $c_{9}\in(0,\infty)$ such that
\begin{equation}
\begin{aligned}\label{Meq13a}
&\E\Bigg[\Bigg|\int\limits_0^t \sum_{i=0}^{n-1} \Bigg( \int\limits_{t_i}^s \alpha(f,v,u) \diff u \Bigg) \mathds{1}_{(t_i, t_{i+1}]}(s) \diff Z(s) \Bigg|^p\Bigg]
\leq c_{9}\,  \delta^p.
\end{aligned}
\end{equation}
For all $(f,Z)\in \{(\sigma,W),(\rho,N)\}$ and all $t\in [0,T]$ we get
\begin{equation}
\begin{aligned}
&\mathbb{E}\Bigg[\Bigg|\int\limits_0^t\sum\limits_{i=0}^{n-1}\Biggl(\int\limits_{t_i}^s\Bigl(\beta(f,t_i,u)-L_1f(U_i)\Bigr)\diff W(u)\Biggr)\mathds{1}_{(t_i,t_{i+1}]}(s)\diff Z(s)\Bigg|^p\Bigg]\\
&\leq \hat c\int\limits_0^t\sum\limits_{i=0}^{n-1}\mathbb{E}\Bigg[\Bigg|\int\limits_{t_i}^s\Bigl(\beta(f,t_i,u)-L_1f(U_i)\Bigr)\diff W(u)\Bigg|^p\Bigg]\cdot\mathds{1}_{(t_i,t_{i+1}]}(s)\diff s.
\end{aligned}
\end{equation}
Further, there exists a constant $c_{10}\in(0,\infty)$ such that for all $s\in [t_i,t_{i+1}]$,
\begin{equation}
 \mathbb{E}\Bigg[\Bigg|\int\limits_{t_i}^s\Bigl(\beta(f,t_i,u)-L_1f(U_i)\Bigr)\diff W(u)\Bigg|^p\Bigg]
 \leq c_{10}(s-t_i)^{\frac{p}{2}-1}\cdot\mathbb{E}\Bigg[\int\limits_{t_i}^s\big|\beta(f,t_i,u)-L_1f(U_i)\big|^p\diff s\Bigg].
\end{equation}
In addition, for $u\in [t_i,t_{i+1}]$,
\begin{equation}
\begin{aligned}
 &|\beta(f,t_i,u)-L_1f(U_i)|\leq |\beta(f,t_i,u)-L_1f(t_i,X(u))|+|L_1f(t_i,X(u))-L_1f(t_i,X^{(\delta)}(t_i))|\\
 &\leq K_3|X(u)-X(t_i)|+K_3|X(t_i)-X^{(\delta)}(t_i)|+K_1^2(1+|X(u)|)\cdot |u-t_i|^{\varrho_2}.
\end{aligned}
\end{equation}
Using \eqref{MomEstSol} and \eqref{MomEstIncSol} we conclude the existence of constants $c_{11},c_{12},c_{13}\in(0,\infty)$ such that
\begin{equation}
 \mathbb{E}\big[|\beta(f,t_i,u)-L_1f(U_i)|^p\Big]
 \leq c_{11} (u-t_i)+c_{12} (u-t_i)^{p\varrho_2}+c_{13}\mathbb{E}\big[  |X(t_i)-X^{(\delta)}(t_i)|^p\big]. 
\end{equation}
Hence, there exist constants $c_{14},c_{15},c_{16}\in(0,\infty)$ such that
\begin{equation}
\begin{aligned}\label{Meq13b}
 &\mathbb{E}\Bigg[\Bigg|\int\limits_0^t\sum\limits_{i=0}^{n-1}\Biggl(\int\limits_{t_i}^s\Bigl(\beta(f,t_i,u)-L_1f(U_i)\Bigr)\diff W(u)\Biggr)\mathds{1}_{(t_i,t_{i+1}]}(s)\diff Z(s)\Bigg|^p\Bigg]\\
&\leq c_{14}\delta^{\frac{p}{2}+1}+c_{15}\delta^{p(\varrho_2+\frac{1}{2})}+c_{16}\int\limits_0^t\sum\limits_{i=0}^{n-1}\mathbb{E}\Big[\big|X(t_i)-X^{(\delta)}(t_i)\big|^p\Big] \cdot\mathds{1}_{(t_i,t_{i+1}]}(s)\diff s.
\end{aligned}
\end{equation}
Further, for $(f,Z)\in \{(\sigma,W),(\rho,N)\}$ and $t\in [0,T]$  we obtain, using Lemma \ref{BDGaKunita} and the fact that the intervals $(t_i,t_{i+1}]$ are disjoint for different $i\in \{0,\ldots, n-1\}$, that
\begin{equation}
\begin{aligned}
 &\mathbb{E}\Bigg[\Bigg|\int\limits_0^t \sum_{i=0}^{n-1} \Biggl(\int\limits_{t_i}^{s-} \Bigl(\gamma(f,t_i,u)-L_{-1} f(U_i)\Bigr)\diff N(u) \Biggr) \mathds{1}_{(t_i,t_{i+1}]}(s) \diff Z(s)\Bigg|^p\Bigg]\\
 &\leq  \hat c\,\mathbb{E}\Bigg[\int\limits_0^t\sum\limits_{i=0}^{n-1}\Bigg|\int\limits_{t_i}^{s-}\Bigl(\gamma(f,t_i,u)-L_{-1}f(U_i)\Bigr)\diff N(u)\Bigg|^p\cdot\mathds{1}_{(t_i,t_{i+1}]}(s)\diff s\Bigg].
\end{aligned}
\end{equation}
Since $N$ has at most finitely many jumps on $[0,T]$, we can replace $s-$ by $s$ in the upper integration limit without affecting the value of the outer integral. Then we apply Lemma \ref{BDGaKunita} to obtain
\begin{equation}
\begin{aligned}\label{Meq13aa}
 &\mathbb{E}\Bigg[\Bigg|\int\limits_0^t \sum_{i=0}^{n-1} \Biggl(\int\limits_{t_i}^{s-} \Bigl(\gamma(f,t_i,u)-L_{-1} f(U_i)\Bigr)\diff N(u) \Biggr) \mathds{1}_{(t_i,t_{i+1}]}(s) \diff Z(s)\Bigg|^p\Bigg]\\
 &\leq  \hat c\,\int\limits_0^t\sum\limits_{i=0}^{n-1}\mathbb{E}\int\limits_{t_i}^{s} \big|\gamma(f,t_i,u)-L_{-1}f(U_i)\big|^p\diff u\cdot\mathds{1}_{(t_i,t_{i+1}]}(s)\diff s.
\end{aligned}
\end{equation}
Using the Lipschitz continuity of $f$ in space and the triangle inequality we obtain that
\begin{equation}
\begin{aligned}
 &|\gamma(f,t_i,u)-L_{-1}f(U_i)|
 \leq 2\, K_1 |X(u-) - X^{(\delta)}(t_i)| +K_1 |\rho(u, X(u-)) - \rho(t_i, X^{(\delta)}(t_i)|.
\end{aligned}
\end{equation}
Adding and subtracting $\rho(t_i, X(u-))$, using the Lipschitz continuity of $\rho$ in space, the Hölder continuity of $\rho$ in time, and adding and subtracting $X(t_i)$ we obtain 
\begin{equation}
\begin{aligned}\label{Meq13c}
 &|\gamma(f,t_i,u)-L_{-1}f(U_i)|\\
 &\leq (2\, K_1 + K_1^2) |X(u-) - X(t_i)| +  (2\, K_1 + K_1^2) |X(t_i) - X^{(\delta)}(t_i)|\\
 &\quad+K_1^2 (1+|X(u-)|)|u-t_i|^{\varrho_3}.
\end{aligned}
\end{equation}
Note that $X(u)$ and $X(u-)$ differ only for at most finitely many $u\in[0,T]$. Hence, we conclude from \eqref{Meq13c} that for all $s\in [t_i,t_{i+1}]$
\begin{equation}
\begin{aligned}
    &\int\limits_{t_i}^s|\gamma(f,t_i,u)-L_{-1}f(U_i)|^p \diff u\\
    &\leq  3^{p-1}(2\, K_1 + K_1^2)^p \int\limits_{t_i}^s|X(u-) - X(t_i)|^p\diff u +  \delta\cdot 3^{p-1}(2\, K_1 + K_1^2)^p |X(t_i) - X^{(\delta)}(t_i)|^p\\
 &\quad+3^{p-1}K_1^{2p} \int\limits_{t_i}^s(1+|X(u-)|)^p|u-t_i|^{p\varrho_3}\diff u\\
 &= 3^{p-1}(2\, K_1 + K_1^2)^p \int\limits_{t_i}^s|X(u) - X(t_i)|^p\diff u +  \delta\cdot 3^{p-1}(2\, K_1 + K_1^2)^p |X(t_i) - X^{(\delta)}(t_i)|^p\\
 &\quad+3^{p-1}K_1^{2p} \int\limits_{t_i}^s(1+|X(u)|)^p|u-t_i|^{p\varrho_3}\diff u\\
 &\leq   3^{p-1}(2\, K_1 + K_1^2)^p \int\limits_{t_i}^s|X(u) - X(t_i)|^p\diff u +  T\cdot 3^{p-1}(2\, K_1 + K_1^2)^p |X(t_i) - X^{(\delta)}(t_i)|^p\\
 &\quad+3^{p-1}K_1^{2p} (1+\sup\limits_{v\in [0,T]}|X(v)|)^p\delta^{p\varrho_3+1}.
\end{aligned}    
\end{equation}
Therefore,
\begin{equation}
\begin{aligned}
&\mathbb{E} \int\limits_{t_i}^s\big|\gamma(f,t_i,u)-L_{-1}f(U_i)\big|^p\diff u\\
&\leq 3^{p-1}  (2\, K_1 + K_1^2)^p \,\int\limits_{t_i}^s\mathbb{E}\Big[ \big|X(u) - X(t_i)\big|^p\Big]\diff u +  T\cdot3^{p-1}(2\, K_1 + K_1^2)^p \,\mathbb{E}\Big[ \big|X(t_i) - X^{(\delta)}(t_i)\big|^p\Big]\\
 &\quad+3^{p-1}K_1^{2p}\, \mathbb{E}\Big[ \big(1+\sup_{v\in[0,T]}|X(v)|\big)^p\Big]\delta^{p\varrho_3+1}.
\end{aligned}
\end{equation}
Combining this, \eqref{MomEstSol}, \eqref{MomEstIncSol}, and \eqref{Meq13aa} we obtain that there exist constants $c_{17},c_{18},c_{19}\in(0,\infty)$ such that
\begin{equation}
\begin{aligned}
 &\mathbb{E}\Bigg[\Bigg|\int\limits_0^t \sum_{i=0}^{n-1} \Biggl(\int\limits_{t_i}^{s-} \Bigl(\gamma(f,t_i,u)-L_{-1} f(U_i)\Bigr)\diff N(u) \Biggr) \mathds{1}_{(t_i,t_{i+1}]}(s) \diff Z(s)\Bigg|^p\Bigg]\\
 &\leq  c_{17}\delta^2+c_{18}\delta^{p(\varrho_3+\frac{1}{p})}+c_{19}\int\limits_0^t\sum\limits_{i=0}^{n-1}\mathbb{E}\Big[|X(t_i)-X^{(\delta)}(t_i)|^p\Big]\cdot\mathds{1}_{(t_i,t_{i+1}]}(s)\diff s.
\end{aligned}
\end{equation}

With these estimates we can calculate the randomized Milstein algorithms' error as follows. There exists a constant $c_{20}\in(0,\infty)$ such that
\begin{equation}
    \begin{aligned}\label{part_est_XXd}
    &\E\Big[|X(t)-X^{(\delta)}_c(t)|^p\Big]\\
    &\leq c_{20}\Big( \E\Big[\big|A(t)-A^{(\delta)}(t)\big|^p\Big] + \E\Big[\big|B(t)-B^{(\delta)}(t)\big|^p\Big] + \E\Big[\big|C(t)-C^{(\delta)}(t)\big|^p\Big]\Big).
    \end{aligned}
\end{equation}
Combining \eqref{Meq6} resp.~\eqref{Meq8} with  \eqref{Meq10}, \eqref{Meq10a}, \eqref{Meq13a}, \eqref{Meq13b}, and \eqref{Meq13c}, we obtain the existence of constants $c_{21},c_{22},c_{23},c_{24}\in(0,\infty)$ such that for all $t\in [0,T]$ hold
\begin{equation}\label{diff_est_BBd}
    \mathbb{E}\Big[\big|B(t)-B^{(\delta)}(t)\big|^p\Big]
    \leq c_{21} \delta^{p\min\{\frac{2}{p},\varrho_2,\varrho_3+\frac{1}{p}\}}+c_{22}\int\limits_0^t\sum\limits_{i=0}^{n-1}\mathbb{E}\Big[\big|X(t_i)-X^{(\delta)}(t_i)\big|^p\Big]\cdot\mathds{1}_{(t_i,t_{i+1}]}(s)\diff s,
\end{equation}
\begin{equation}\label{diff_est_CCd}
    \mathbb{E}\Big[\big|C(t)-C^{(\delta)}(t)\big|^p\Big] 
    \leq c_{23} \delta^{p\min\{\frac{2}{p},\varrho_3,\varrho_2+\frac{1}{2}\}}+c_{24}\int\limits_0^t\sum\limits_{i=0}^{n-1}\mathbb{E}\Big[\big|X(t_i)-X^{(\delta)}(t_i)\big|^p\Big]\cdot\mathds{1}_{(t_i,t_{i+1}]}(s)\diff s.
\end{equation}
Some of the terms in \eqref{Meq3} and \eqref{Meq4} still need to be estimated. $\mathbb{E}[|\tilde M_1^{(\delta)}(t)|^p]$ is already considered in \eqref{Meq13a}. 
By following the same procedure as in \cite[pages 8--10]{Morkisz2020} and using Lemma \ref{MEstL}, we obtain that there exists a constant $c_{25}\in(0,\infty)$ such that for all $t\in [0,T]$,
\begin{equation}
  \label{est_tm2}
  \mathbb{E}\Big[\big|\tilde M^{(\delta)}_2(t)\big|^p\Big]\leq c_{25}\delta^{p\min\{\frac{1}{2}+\varrho_1,1\}}.
\end{equation}
For $\mathbb{E}[|\tilde M^{(\delta)}_3(t)|^p]$ there exists a constant $c_{26}\in(0,\infty)$ such that for all $t\in [0,T]$ there exists $\ell\in\{0,1,\ldots,n-1\}$ with $t\in [t_\ell,t_{\ell+1}]$ and
\begin{equation}
\begin{aligned}\label{mainM3p1}
  &\mathbb{E}\Big[\big|\tilde M^{(\delta)}_3(t)\big|^p\Big]
  \leq c_{26}\Bigg( \mathbb{E}\Bigg[\Bigg|\int\limits_0^{t_\ell}\sum\limits_{i=0}^{n-1}\Bigg(\int\limits_{t_i}^s(\gamma(\mu,s,u)-\gamma(\mu,t_i,u))\diff N(u)\Bigg)\cdot\mathds{1}_{(t_i,t_{i+1}]}(s)\diff s\Bigg|^p\Bigg]\\
  &\quad\quad\quad\quad \quad\quad\quad\quad \quad +\mathbb{E}\Bigg[\Bigg|\int\limits_0^{t_\ell}\sum\limits_{i=0}^{n-1}\Bigg(\int\limits_{t_i}^s\gamma(\mu,t_i,u)\diff \tilde N(u)\Bigg)\cdot\mathds{1}_{(t_i,t_{i+1}]}(s)\diff s\Bigg|^p\Bigg]\\
  &\quad\quad\quad\quad \quad\quad\quad\quad \quad+\lambda^p\mathbb{E}\Bigg[\Bigg|\int\limits_0^{t_\ell}\sum\limits_{i=0}^{n-1}\Bigg(\int\limits_{t_i}^s\gamma(\mu,t_i,u)\diff u\bigg)\cdot\mathds{1}_{(t_i,t_{i+1}]}(s)\diff s\Bigg|^p\Bigg]\\
  &\quad\quad\quad\quad \quad\quad\quad\quad \quad+\mathbb{E}\Bigg[\Bigg|\int\limits_{t_\ell}^t\Bigg(\int\limits_{t_\ell}^s\gamma(\mu,s,u)\diff  N(u)\Bigg)\diff s\Bigg|^p\Bigg]\Big).\\
\end{aligned}
\end{equation}
By Lemmas \ref{BDGaKunita} and \ref{MEstL} we get the existence a constant $c_{27}\in(0,\infty)$ such that
\begin{equation}
\begin{aligned}\label{mainM3p2}
 & \mathbb{E}\Bigg[\Bigg|\int\limits_0^{t_\ell}\sum\limits_{i=0}^{n-1}\Bigg(\int\limits_{t_i}^s(\gamma(\mu,s,u)-\gamma(\mu,t_i,u))\diff N(u)\Bigg)\cdot\mathds{1}_{(t_i,t_{i+1}]}(s)\diff s\Bigg|^p\Bigg]\\
 &\leq \hat c \sum\limits_{i=0}^{n-1}\int\limits_{t_i}^{t_{i+1}}\mathbb{E}\Bigg[\int\limits_{t_i}^s|\gamma(\mu,s,u)-\gamma(\mu,t_i,u)|^p\diff u\Bigg]\diff s\\
 &\leq \hat c K_7\sum\limits_{i=0}^{n-1}\int\limits_{t_i}^{t_{i+1}}\mathbb{E}\Bigg[\int\limits_{t_i}^s(1+|X(u-)|)^p\cdot(s-t_i)^{p\varrho_1}\diff u\Bigg]\diff s\leq c_{27}\delta^{p(\varrho_1+\frac{1}{p})}.
\end{aligned}
\end{equation}
The H\"older inequality, \eqref{MomEstSol}, and Lemma \ref{MEstL} are used in a similar manner as above to obtain that there exists a constant $c_{28}\in(0,\infty)$ such that for all $t\in[t_\ell,t_{\ell+1}]$,
\begin{equation}
\begin{aligned}\label{mainM3p3}
&\E\Bigg[\Bigg|\int\limits_{t_\ell}^{t}  \Bigg( \int\limits_{t_\ell}^s  \gamma(\mu,s,u)\diff  N(u)\Bigg) \diff s\Bigg|^p\Bigg] 
\leq \E\Bigg[ \Bigg( \int\limits_{t_\ell}^{t} \Bigg| \int\limits_{t_\ell}^s \gamma(\mu,s,u)\diff  N(u) \Bigg| \diff s\Bigg)^p\Bigg]\\
&\leq \delta^{p-1}\, \int\limits_{t_\ell}^{t_{\ell+1}}\E\Bigg[ \Bigg| \int\limits_{t_\ell}^s  \gamma(\mu,s,u)\diff  N(u) \Bigg|^p\Bigg] \diff s
\leq \hat c\delta^{p-1}\, \int\limits_{t_\ell}^{t_{\ell+1}}\E\Bigg[\int\limits_{t_\ell}^s  |\gamma(\mu,s,u)|^p\diff  u\Bigg] \diff s\leq c_{28} \delta^{p+1}.
\end{aligned}
\end{equation}
Further, we obtain that there exists a constant $c_{29}\in(0,\infty)$ such that
\begin{equation}
\begin{aligned}\label{mainM3p4}
&\E\Bigg[\Bigg|\int\limits_0^{t_\ell}\sum\limits_{i=0}^{n-1}\Bigg(\int\limits_{t_i}^s\gamma(\mu,t_i,u)\diff u\Bigg)\cdot\mathds{1}_{(t_i,t_{i+1}]}(s)\diff s\Bigg|^p\Bigg] \\
&\leq T^{p-1}\sum\limits_{i=0}^{n-1}\int\limits_{t_i}^{t_{i+1}}\E\Biggl[\Biggl|\int\limits_{t_i}^s\gamma(\mu,t_i,u)\diff u\Biggl|^p\Biggr]\diff s
\leq T^{p-1}\delta^{p-1}\sum\limits_{i=0}^{n-1}\int\limits_{t_i}^{t_{i+1}}\E\Biggl[\int\limits_{t_i}^s|\gamma(\mu,t_i,u)|^p\diff u\Biggr]\diff s\leq c_{29} \delta^{p}.
\end{aligned}
\end{equation}
In addition, it holds
\begin{equation}\label{mainM3p5}
    \mathbb{E}\Bigg[\Biggl|\int\limits_0^{t_\ell}\sum\limits_{i=0}^{n-1}\Bigg(\int\limits_{t_i}^s\gamma(\mu,t_i,u)\diff \tilde N(u)\Bigg)\cdot\mathds{1}_{(t_i,t_{i+1}]}(s)\diff s\Bigg|^p\Bigg]=\E\Big[\big|\tilde Z_{\ell-1}\big|^p\Big],
\end{equation}
where
\begin{equation}
    \tilde Z_k=\sum\limits_{i=0}^k \tilde Y_i, \quad k\in\{0,1,\ldots,n-1\},
\end{equation}
with $Z_{-1}=0$ and
\begin{equation}
    \tilde Y_i=\int\limits_{t_i}^{t_{i+1}}\Bigl(\int\limits_{t_i}^{s}\gamma(\mu,t_i,u)\diff \widetilde N(u)\Bigr)\diff s.
\end{equation}
Let $\mathcal{G}_k:=\mathcal{F}_{t_{k+1}}$. Then it holds that $\{\tilde Z_k,\mathcal{G}_k\}_{k=0,1,\ldots,n-1}$ is a discrete-time martingale, since $\widetilde Z_k$ is adapted to $\mathcal{G}_k$ for $k\in\{0,\ldots,n-1\}$ and Fubini's theorem for conditional expectations, e.g.~\cite{Brooks72}, implies
\begin{equation}
\begin{aligned}
\E\big[\widetilde Z_{k+1} - \widetilde Z_k \big|\mathcal{G}_{k}\big] &=
\E\Bigg[ \int\limits_{t_{k+1}}^{t_{k+2}} \Bigg( \int\limits_{t_{k+1}}^s  \gamma(\mu, t_{k+1}, u) \diff \widetilde N(u) \Bigg) \diff s\Bigg|\mathcal{F}_{t_{k+1}} \Bigg]\\
&= \int\limits_{t_{k+1}}^{t_{k+2}} \E\Bigg[  \int\limits_{t_{k+1}}^s  \gamma(\mu, t_{k+1}, u) \diff \widetilde N(u)  \Bigg|\mathcal{F}_{t_{k+1}} \Bigg]\diff s
=0.
\end{aligned}
\end{equation}
As a result, we conclude from the discrete version of the Burkholder-Davis-Gundy inequality and Jensen's inequality that there exist constants $c_{30},c_{31}\in(0,\infty)$ such that 
\begin{equation}\label{mainM3p6}
    \E\Big[\big|\tilde Z_k\big|^p\Big]\leq c_{30}\E\Bigl[\Bigl(\sum\limits_{i=0}^k|\tilde Y_i|^2\Bigr)^{p/2}\Bigr]\leq c_{30} n^{\frac{p}{2}-1}\sum\limits_{i=0}^{n-1}\E\Big[\big|\tilde Y_i\big|^p\Big]\leq c_{31}\delta^{\frac{p}{2}+1},
\end{equation}
for $k\in\{0,1,\ldots,n-1\}$. Combining \eqref{mainM3p1}, \eqref{mainM3p2}, \eqref{mainM3p3}, \eqref{mainM3p4}, \eqref{mainM3p5}, and \eqref{mainM3p6} we get that there exists a constant $c_{32}\in(0,\infty)$ such that for all $t\in [0,T]$,
\begin{equation}\label{est_tm3}
    \E\Big[\big|\tilde M_3(t)\big|^p\Big]\leq c_{32}\delta^{p\min\{\varrho_1+\frac{1}{p},\frac{1}{2}+\frac{1}{p},1\}}.
\end{equation}
By \eqref{Meq13a}, \eqref{est_tm2}, and \eqref{est_tm3} we obtain that there exists a constant $c_{33}\in(0,\infty)$ such that
\begin{equation}
\label{est_ta1}
\E\Big[\big|\tilde A_1^{(\delta)}(t)\big|^p\Big]\leq c_{33} \delta^{p\min\{\varrho_1+\frac{1}{p},\frac{1}{2}+\frac{1}{p},1\}}.
\end{equation}
Analog to the proof of \cite[equation 33]{Morkisz2020} it can be shown that there exists a constant $c_{34}\in(0,\infty)$ such that for all $t\in [0,T]$ holds
\begin{equation}\label{est_ta2}
    \mathbb{E}\Big[\big|\tilde A^{(\delta)}_2(t)\big|^p\Big]\leq c_{34}\delta^{p(\varrho_1+\frac{1}{2})}.
\end{equation}
Additionally, we estimate $\mathbb{E}[|\tilde A_3^{(\delta)}(t)|^p]$ using \eqref{Meq10}. Therefore, by \eqref{Meq10}, \eqref{est_ta1}, and \eqref{est_ta2} we obtain that there exist constants $c_{35},c_{36}\in(0,\infty)$ such that for all $t\in [0,T]$ holds
\begin{equation}
\label{diff_est_AAd}
    \E\Big[\big|A(t)-A^{(\delta)}(t)\big|^p\Big]\leq c_{35}\delta^{p\min\{\varrho_1+\frac{1}{p},\frac{1}{2}+\frac{1}{p},1\}}+c_{36}\int\limits_0^t\sum\limits_{i=0}^{n-1}\mathbb{E}\Big[\big|X(t_i)-X^{(\delta)}(t_i)\big|^p\Big]\cdot\mathds{1}_{(t_i,t_{i+1}]}(s)\diff s.
\end{equation}
Using \eqref{part_est_XXd}, \eqref{diff_est_BBd},  \eqref{diff_est_CCd}, and \eqref{diff_est_AAd} implies the existence of constants $c_{37},c_{38}\in(0,\infty)$ such that for all $t\in [0,T]$ holds
\begin{equation}
    \begin{aligned}
    &\mathbb{E}\Big[\big|X(t)-X^{(\delta)}_c(t)\big|^p\Big]
    \leq c_{37}\delta^{p\min\{\frac{2}{p},\varrho_1+\frac{1}{p},\varrho_2,\varrho_3\}}+c_{38}\int\limits_0^t\sum\limits_{i=0}^{n-1}\mathbb{E}\Big[\big|X(t_i)-X^{(\delta)}(t_i)\big|^p\Big]\cdot\mathds{1}_{(t_i,t_{i+1}]}(s)\diff s\\
    &\leq c_{37}\delta^{p\min\{\frac{2}{p},\varrho_1+\frac{1}{p},\varrho_2,\varrho_3\}}+c_{38}\int\limits_0^t\sup\limits_{0\leq u\leq s}\mathbb{E}\Big[\big|X(u)-X^{(\delta)}_c(u)\big|^p\Big]\diff s,
    \end{aligned}
\end{equation}
and hence,
\begin{equation}
\label{final_err_est}
    \sup\limits_{0\leq u\leq t}\mathbb{E}\Big[\big|X(u)-X^{(\delta)}_c(u)\big|^p\Big]
    \leq c_{37}\delta^{p\min\{\frac{2}{p},\varrho_1+\frac{1}{p},\varrho_2,\varrho_3\}}+c_{38}\int\limits_0^t\sup\limits_{0\leq u\leq s}\mathbb{E}\Big[\big|X(u)-X^{(\delta)}_c(u)\big|^p\Big]\diff s.
\end{equation}
Since \eqref{MomEstSol} and \eqref{MomEst} guarantee that $\displaystyle{
[0,T]\ni t\mapsto \sup\limits_{0\leq u\leq t} \E\Big[\big|X(u)-X^{(\delta)}_c(u)\big|^p\Big]\in[0,\infty)}$ is bounded and non-decreasing, it is Borel measurable. Therefore, applying Grownall's lemma to \eqref{final_err_est} shows that there exists a constant $C\in(0,\infty)$ such that for all $t\in[0,T]$, 
\begin{equation}
\begin{aligned}\label{Meq45}
\sup_{0\leq u\leq t} \E\Big[\big|X(u)-X^{(\delta)}_c(u)\big|^p\Big]\leq C\delta^{p\min\{\frac{2}{p},\varrho_1+\frac{1}{p},\varrho_2,\varrho_3\}}.
\end{aligned}
\end{equation}
\end{proof}
%%%%%%%%%%
\begin{remark}
Note that for the classical Milstein scheme $\bar X^{(\delta)}$, which is defined by \eqref{MS} when $\xi_i$ is replaced by $t_i$ for all $i\in\{0,1,\ldots,n-1\}$, there exists a constant $K_9\in(0,\infty)$ such that for all $n\in\N$,
\begin{equation}
\begin{aligned}
\sup_{0\leq t\leq T} \|X(t) - \bar X^{(\delta)}_c(t)\|_{L^p(\Omega)}
\leq K_9\delta^{\min\{\frac{2}{p},\varrho_1,\varrho_2,\varrho_3\}}.
\end{aligned}
\end{equation}
This follows from a straightforward modification of the proof of Theorem \ref{ThmUpperBound} and indicates that the convergence rate of the Milstein scheme is improved by randomization.
\end{remark}
%%%%%%%%%%%
\begin{remark}
In the jump-free case ($\rho=0$) we obtain following the proof of Theorem  \ref{ThmUpperBound} that
there exists a constant $K_{10}\in (0,\infty)$ such that for all $n\in\N$,
\begin{equation}
\begin{aligned}
\sup_{0\leq t\leq T} \|X(t) - X^{(\delta)}_c(t)\|_{L^p(\Omega)}
\leq  K_{10}\delta^{\min\{\varrho_1+\frac{1}{2},\varrho_2\}}.
\end{aligned}
\end{equation}
Thus, we obtain the same upper error bound for the randomized Milstein process as in \cite[Proposition 1]{Morkisz2020} under slightly weaker assumptions on $\mu$ and $\sigma$. Additionally, we recover for $\varrho_2=\min\{\frac{1}{2}+\varrho_1,1\}$ the upper error bound found in \cite{kruse2018} for a two-stage randomized Milstein scheme.
\end{remark}
%%%%%%%%%%%

\section{Lower bounds and optimality}

In this section, we provide lower error bounds and optimality results in the IBC framework, \cite{TWW88}. We set $p=2$ and assume only standard information is available, i.e.~a finite number of point evaluations of $W$ and $N$.
First, we look at the approximation of scalar SDEs which satisfy the JCC. Afterwards, we study the multidimensional case.
%%%%%%%%%%%%%%
\subsection{Scalar case and optimality of the randomized Milstein algorithm}

We provide lower error bound and optimality results for the randomized Milstein algorithm. We assume $p=2$ and the JCC is satisfied, i.e.
\begin{equation}
\label{JCC}
    L_{-1}\sigma(t,y)=L_1\rho(t,y), \ (t,y)\in [0,T]\times\mathbb{R},
\end{equation}
e.g., \cite{PBL2010}. Under this condition, the randomized Milstein algorithm  only uses standard discrete information about $W$, $N$, i.e.~the values $W(t_1),\ldots,W(t_n),N(t_1),\ldots,N(t_n)$. By \eqref{JCC_int} and \eqref{MS} the scheme simplifies to
\begin{equation}
\begin{aligned}\label{MS_JCC}
&X^{(\delta)}(t_0) = X_0,\\
&X^{(\delta)}(t_{i+1}) = X^{(\delta)}(t_{i}) + \mu(\xi_i,X^{(\delta)}(t_{i})) \delta + \sigma(t_{i},X^{(\delta)}(t_{i}))\Delta W_i +\rho(t_{i},X^{(\delta)}(t_{i})) \Delta N_i\\
&\quad\quad\quad\quad \quad\quad+ L_{1}\sigma(t_{i},X^{(\delta)}(t_{i}))I_{t_i,t_{i+1}}(W,W)
+ L_{-1}\rho(t_{i},X^{(\delta)}(t_{i})) I_{t_i,t_{i+1}}(N,N)\\
&\quad\quad\quad\quad \quad\quad+L_{-1}\sigma(t_{i},X^{(\delta)}(t_{i})) \Delta W_i\Delta N_i, \quad\quad \hbox{for } i \in\{0,\ldots, n-1\}.
\end{aligned}
\end{equation}
Hence, if the JCC is assumed then randomized Milstein algorithm is implementable. 

We define the following function classes to provide worst-case error bounds and optimality analyses. For $K\in(0,\infty)$ and $\gamma\in (0,1]$, a function $f:[0,T] \times \mathbb{R} \to \mathbb{R}$ belongs to the function class $F_K^{\gamma}$ if and only if it satisfies for all $t,s\in[0,T]$ and all $y,z\in\mathbb{R}$ 
\begin{itemize}
	\item [(i)] $f\in C^{0,1}\left([0,T]\times\mathbb{R}\right)$,
	\item [(ii)] $|f(0,0)| \leq K$,
	\item [(iii)] $|f(t,y) - f(t,z)| \leq K |y-z|$,
	\item [(iv)] $|f(t,y) - f(s,y)| \leq K(1+|y|)|t-s|^{\gamma}$,
	\item [(v)] $\left| \frac{\partial f}{\partial y}(t,y) - \frac{\partial f}{\partial y}(t,z) \right| \leq K |y-z|$.
\end{itemize}
Here we consider drift coefficients $\mu$ from the class
\begin{displaymath}
	\mathcal{M}^{\varrho_1}_K=\Biggl\{\mu \in F^{\varrho_1}_K \colon \left| \frac{\partial \mu}{\partial y}(t,y) - \frac{\partial \mu}{\partial y}(s,y) \right| \leq K (1+|y|)|t-s|^{\varrho_1} \ \hbox{for all} \ t,s\in [0,T], y\in\mathbb{R}\Biggr\}.
\end{displaymath}
We assume that the diffusion and jump coefficients $(\sigma,\rho)$ are from the class
\begin{equation}
    \begin{aligned}
    &\mathcal{B}^{\varrho_2,\varrho_3}_K = \Bigl\{ (\sigma,\rho) \in F^{\varrho_2}_K\times F^{\varrho_3}_K \colon |L_1 \sigma(t,y) - L_1 \sigma(t,z)|\leq K|y-z|, \\
    &|L_1 \rho(t,y) - L_1 \rho(t,z)|\leq K|y-z|, \  L_{-1}\sigma(t,y)=L_1\rho(t,y), \ \ \hbox{for all} \ t\in [0,T], y,z\in\mathbb{R}\Bigr\}.
    \end{aligned}
\end{equation}
Recall that $L_1f(t,y)=\sigma(t,y)\frac{\partial f}{\partial y}(t,y)$ and $L_{-1}f(t,y)=f(t,y+\rho(t,y))-f(t,y)$. In addition, define for all $p\in[2,\infty)$,
\begin{displaymath}
	\mathcal{J}^p_{K}=\{X_0\colon\Omega\to\mathbb{R} \colon X_0 \ \hbox{is} \ \mathcal{F}_0-\hbox{measurable}, \mathbb{E}[|X_0|^{2p}]\leq K\}.
\end{displaymath}
The class of input data $(\mu,\sigma,\rho,X_0)$ is defined by 
\begin{equation}
    \mathcal{F}(\varrho_1,\varrho_2,\varrho_3,p,K)=\mathcal{M}^{\varrho_1}_K\times\mathcal{B}^{\varrho_2,\varrho_3}_K\times\mathcal{J}^p_{K}.
\end{equation}
We call $\varrho_1,\varrho_2,\varrho_3,p,K,T$ the parameters of the class $\mathcal{F}(\varrho_1,\varrho_2,\varrho_3,p,K)$.

Next, we define the model of computation. An information vector has the form
\begin{equation}
    \begin{aligned}
	\mathcal{N}(\mu, \sigma,\rho,X_0, W,N)=&[\mu (\xi_0,y_0),\ldots,\mu(\xi_{k_1-1},y_{k_1-1}),
    \sigma(t_0,y_0), \ldots, \sigma(t_{k_1-1},y_{k_1-1}), \\
    & \ \ \rho(t_0,y_0), \ldots, \rho(t_{k_1-1},y_{k_1-1}), 
    \frac{\partial\sigma}{\partial y}(t_0,y_0),\ldots, \frac{\partial\sigma}{\partial y}(t_{k_1-1},y_{k_1-1}),\\
    & \ \ \sigma(t_0, z_0),\ldots, \sigma(t_{k_1-1}, z_{k_1-1}),
    \rho(t_0, v_0),\ldots, \rho(t_{k_1-1},v_{k_1-1}),\\
    &\ \ W(s_0),\ldots, W(s_{k_2-1}),
    N(q_0), \ldots, N(q_{k_3-1}),X_0],
    \end{aligned}
\end{equation}
where $k_1$, $k_2$, $k_3\in\N$ and $[\xi_0,\xi_1,\ldots, \xi_{k_1-1}]$ is a random vector on $(\Omega,\F, \mathbb{P})$ with  values in $[0,T]^{k_1}$. We assume that the sigma-field generated by $\xi_0,\xi_1,\ldots, \xi_{k_1-1}$ is independent of $\F_T$. Moreover, $t_0, t_1,\ldots, t_{k_1-1} \in [0,T]$, $s_0, s_1,\ldots, s_{k_2-1}\in [0,T]$, and $q_0, q_1,\ldots, q_{k_3-1}\in [0,T]$ 
are given time points. We assume that $ s_i \neq s_j$, $q_i\neq q_j$ for all $i \neq j$. The evaluation points $y_j, z_j, v_j$ for the spatial variables of $\mu, \sigma$, $\partial \sigma/\partial y$, and $\rho$ are given in an adaptive way with respect to $(\mu, \sigma,\rho, X_0)$ and the standard discrete information about $W$ and $N$. This means that for some measurable mappings $\psi_j$, $j\in\{0,1,\ldots,k_1-1\}$, it holds that
\begin{equation}
    (y_0,z_0,v_0)=\psi_0 (W(s_0),\ldots, W(s_{k_2-1}),N(q_0), \ldots, N(q_{k_3-1}),X_0)
\end{equation}
and 
\begin{equation}
    \begin{aligned}
    &(y_j,z_j,v_j)=\psi_j(\mu (\xi_0,y_0),\ldots,\mu(\xi_{j-1},y_{j-1}),
	\sigma(t_0,y_0), \ldots, \sigma(t_{j-1},y_{j-1}),\\
    & \quad\quad\quad\quad\quad\quad\quad\quad \rho(t_0,y_0), \ldots, \rho(t_{j-1},y_{j-1}), 
    \frac{\partial\sigma}{\partial y}(t_0,y_0),\ldots, \frac{\partial\sigma}{\partial y}(t_{j-1},y_{j-1}),\\
    &\quad\quad\quad\quad\quad\quad\quad\quad  \sigma(t_0, z_0),\ldots, \sigma(t_{j-1}, z_{j-1}),
    \rho(t_0, v_0),\ldots, \rho(t_{j-1},v_{j-1}),\\
    & \quad\quad\quad\quad\quad\quad\quad\quad W(s_0),\ldots, W(s_{k_2-1}),
    N(q_0), \ldots, N(q_{k_3-1}),X_0).
    \end{aligned}
\end{equation}
The total number of evaluations of $\mu,\sigma,\rho$, $W$, and $N$ is given by $l=6k_1+k_2+k_3$.

Any algorithm $\mathcal{A}$ that computes an approximation to $X(T)$ using the information $\mathcal{N}(\mu, \sigma,\rho,X_0, W,N)$ is of the form
\begin{equation}
\label{alg_def_1}
    \mathcal{A}(\mu, \sigma,\rho,X_0, W,N)=\varphi(\mathcal{N}(\mu, \sigma,\rho,X_0, W,N)),
\end{equation}
where $\varphi:\mathbb{R}^{3k_1+k_2+k_3+1}\to\mathbb{R}$ is a Borel measurable function. For a fixed $n\in\mathbb{N}$ we denote by $\Phi_n$ the class of all algorithms \eqref{alg_def_1} with total number of evaluations $l\leq n$.

For $(\mu,\sigma,\rho,X_0)\in\mathcal{F}(\varrho_1,\varrho_2,\varrho_3,p,K)$ we define the error of $\mathcal{A}\in\Phi_n$ as 
\begin{equation}
    e^{(2)}(\mathcal{A},\mu, \sigma,\rho,X_0, W,N)=\|\mathcal{A}(\mu, \sigma,\rho,X_0, W,N)-X(\mu, \sigma,\rho,X_0)(T)\|_2.
\end{equation}
The worst-case error of $\mathcal{A}$ in a subclass $\mathcal{G}$ of $\mathcal{F}(\varrho_1,\varrho_2,\varrho_3,p,K)$ is defined by
\begin{equation}
    e^{(2)}(\mathcal{A},\mathcal{G},W,N)=\sup\limits_{(\mu, \sigma,\rho,X_0)\in\mathcal{G}}e^{(2)}(\mathcal{A},\mu, \sigma,\rho,X_0, W,N),
\end{equation}
while the $n$-th minimal error in $\mathcal{G}$ is 
\begin{equation}
    e^{(2)}_n(\mathcal{G},W,N)=\inf\limits_{\mathcal{A}\in\Phi_n}e^{(2)}(\mathcal{A},\mathcal{G},W,N).
\end{equation}
The aim is to find sharp bounds for $e^{(p)}_n(\mathcal{F}(\varrho_1,\varrho_2,\varrho_3,p,K),W,N)$, i.e.~lower and upper error bounds which are equal up to constants.

The randomized Milstein algorithm can be written as 
\begin{equation}
    \mathcal{A}^{RM}_n(\mu,\sigma,\rho,X_0,W,N)=X^{(\delta)}(T),
\end{equation}
where $X^{(\delta)}(T)$ is defined in \eqref{MS}. It holds that  $\mathcal{A}^{RM}_n\in\Phi_{8n}$.
%%%%%%%%%%%%%%%%%%%
\begin{theorem}
\label{opt_rm_scalar}
    It holds that
    \begin{equation}
        e^{(2)}_n(\mathcal{F}(\varrho_1,\varrho_2,\varrho_3,2,K),W,N)=\Theta(n^{-\min\{\varrho_1+\frac{1}{2},\varrho_2,\varrho_3\}})
    \end{equation}
    as $n\to +\infty$.
\end{theorem}
\begin{proof} Since $e^{(2)}_n(\mathcal{F}(\varrho_1,\varrho_2,\varrho_3,2,K),W,N)\leq e^{(2)}(\mathcal{A}^{RM}_n,\mathcal{F}(\varrho_1,\varrho_2,\varrho_3,2,K),W,N)$, Theorem \ref{ThmUpperBound} implies the upper bound  $O(n^{-\min\{\varrho_1+\frac{1}{2},\varrho_2,\varrho_3\}})$ for $e^{(2)}_n(\mathcal{F}(\varrho_1,\varrho_2,\varrho_3,2,K),W,N)$.

For the lower bounds let $\mathcal{A}$ be any algorithm from $\Phi_n$ that uses at most $n$ evaluations of $(\mu,\sigma,\rho)$, $W$, and $N$. We consider the following subclasses of  $\mathcal{F}(\varrho_1,\varrho_2,\varrho_3,2,K)$:
\begin{equation*}
    \mathcal{G}_{1}(\varrho_1,1,1,2,K) =\mathcal{\bar M}_{K}^{\varrho_1} \times \{(0,0)\} \times \{0\},
\end{equation*}
where
\begin{equation*}
    \mathcal{\bar M}_{K}^{\varrho_1}=\{ \mu\in \mathcal{M}_{K}^{\varrho_1} \ | \ \mu(t,x) = \mu(t,0) \text{ for all } t \in [0,T], x \in \mathbb{R} \},
\end{equation*}
and
\begin{equation*}
    \mathcal{G}_{2}(1,\varrho_2,1,2,K)=\{0\} \times \mathcal{\bar B}^{\varrho_2,1}_K \times \{0\}, 
\end{equation*}
where
\begin{equation*}
\mathcal{\bar B}^{\varrho_2,1}_K=\Bigl\{ (\sigma,0) \in \mathcal{B}^{\varrho_2,1}_K \,\Bigl|\, \sigma(t,y)=\sigma(t,0) \ \hbox{for all} \ t\in [0,T], y\in\mathbb{R}\Bigr\},
\end{equation*}
and
\begin{equation*}
    \mathcal{G}_{3}(1,1,\varrho_3,2,K)=\{0\} \times \mathcal{\bar B}^{1,\varrho_3}_K \times \{0\}, 
\end{equation*}
where
\begin{equation*}
\mathcal{\bar B}^{1,\varrho_3}_K=\Bigl\{ (0,\rho) \in \mathcal{B}^{1,\varrho_3}_K \,\Bigl|\, \rho(t,y)=\rho(t,0) \ \hbox{for all} \ t\in [0,T], y\in\mathbb{R}\Bigr\}.
\end{equation*}
For $(\mu,\sigma,\rho,X_0) \in \mathcal{G}_{1}(\varrho_1,1,1,2,K)$ holds $\displaystyle{X(\mu,\sigma,\rho,X_0)(T)=\int_{0}^{T}\mu(t,0)\diff t}$. Since $k_1=O(n)$ by \cite[Section 2.2.9, Proposition 2]{Novak1988} we obtain that
\begin{equation}
    e(\mathcal{A},\mathcal{G}_{1}(\varrho_1,1,1,2,K))=\Omega(n^{-(\varrho_1+\frac{1}{2})}).
\end{equation}
Further, for $(\mu,\sigma,\rho,X_0) \in \mathcal{G}_{2}(1,\varrho_2,1,2,K)$ holds $\displaystyle{X(\mu,\sigma,\rho,X_0)(T)=\int_{0}^{T}\sigma(t,0)\diff W(t)}$. Since $k_2=O(n)$, \cite[Proposition 5.1(i)]{PMPP2014} gives
\begin{equation}
    e(\mathcal{A},\mathcal{G}_{2}(1,\varrho_2,1,2,K))=\Omega(n^{-\varrho_2}).
\end{equation}
Finally, for $(\mu,\sigma,\rho,X_0) \in \mathcal{G}_{3}(1,1,\varrho_3,2,K)$ holds $\displaystyle{X(\mu,\sigma,\rho,X_0)(T)=\int_{0}^{T}\rho(t,0)\diff N(t)}$. Since $k_3=O(n)$, \cite[Lemma 6]{Przybylowicz2021} yields
\begin{equation}
    e(\mathcal{A},\mathcal{G}_{3}(1,1,\varrho_3,2,K))=\Omega(n^{-\varrho_3}).
\end{equation}
Due to the fact that $\mathcal{G}_{1}(\varrho_1,1,1,2,K)\cup\mathcal{G}_{2}(1,\varrho_2,1,2,K)\cup\mathcal{G}_{3}(1,1,\varrho_3,2,K)\subset\mathcal{F}(\varrho_1,\varrho_2,\varrho_3,2,K)$, we obtain
\begin{equation}
    \begin{aligned}
    &e(\mathcal{A},\mathcal{F}(\varrho_1,\varrho_2,\varrho_3,2,K))\\
    &\geq\max\{e(\mathcal{A},\mathcal{G}_{1}(\varrho_1,1,1,2,K)), e(\mathcal{A},\mathcal{G}_{2}(1,\varrho_2,1,2,K)),e(\mathcal{A},\mathcal{G}_{3}(1,1,\varrho_3,2,K))\}\\
    &=\Omega(n^{-\min\{\varrho_1+\frac{1}{2},\varrho_2,\varrho_3\}}).
    \end{aligned}
\end{equation}
Together with the upper bound, this proves the claim.
\end{proof}
%%%%%%%%%%%%%%%%%%
\begin{remark} 
\label{gain_rm}
For $\varrho_2=\varrho_3=1$ and $\varrho_1\in (1/2,1)$ we compare the worst case errors for the classical Euler--Maruyama algorithm $\mathcal{A}^E_n$, randomized Euler--Maruyama algorithm $\mathcal{A}^{RE}_n$, classical Milstein algorithm $\mathcal{A}^M_n$, and randomized Milstein algorithm $\mathcal{A}^{RE}_n$ in the class $\mathcal{F}(\varrho_1,\varrho_2,\varrho_3,2,K)$. It holds that
\begin{equation}
\begin{aligned}
    &e^{(2)}(\mathcal{A}^E_n,\mathcal{F}(\varrho_1,\varrho_2,\varrho_3,2,K),W,N)=O(n^{-1/2}), \
    e^{(2)}(\mathcal{A}^{RE}_n,\mathcal{F}(\varrho_1,\varrho_2,\varrho_3,2,K),W,N)=O(n^{-1/2}),\\
    &e^{(2)}(\mathcal{A}^M_n,\mathcal{F}(\varrho_1,\varrho_2,\varrho_3,2,K),W,N)=O(n^{-\varrho_1}), \
    e^{(2)}(\mathcal{A}^{RM}_n,\mathcal{F}(\varrho_1,\varrho_2,\varrho_3,2,K),W,N)=O(n^{-1}).
    \end{aligned}
\end{equation}
Consequently, the randomized Milstein algorithm outperforms the other (classical) algorithms in this setting.
\end{remark}
%%%%%%%%%%%%%%%%%%%%%%%
\subsection{Multidimensional case and optimality of the Euler--Maruyama algorithm}

In this section, we discuss lower error bounds for approximating solutions of systems of jump-diffusion SDEs if only standard information about $W$ and $N$ is available. In order to establish suitable lower bounds we extend results from \cite{Clark1980} and analyse the following jump-diffusion L\'evy's area
\begin{equation}
    \begin{aligned}\label{LA1}
    &J(N,W)=I_{0,T}(N,W) = \int\limits_0^T\int\limits_0^{t-} \diff N(s) \diff W(t) = \int\limits_0^T N(t-)\diff W(t)= \int\limits_0^T N(t)\diff W(t).
    \end{aligned}
\end{equation}
The last equality holds because $W$ is continuous and $N(\cdot)$ and $N(\cdot-)$ differ at most in finitely many points. It is essential to note that $J(N,W)=X(T)$, where $X$ is the solution of the two-dimensional SDE
\begin{equation}
    \begin{aligned}\label{2dim_sde}
    & dY(t)=\diff N(t),\\
    & dX(t)= Y(t)\diff W(t), \ t\in [0,T].
    \end{aligned}
\end{equation}

We consider an arbitrary algorithm of the form 
\begin{equation}
    \begin{aligned}\label{LA4}
    \mathcal{A}_n(N,W) = \varphi_n(\mathcal{N}_n(N,W))
    \end{aligned}
\end{equation}
to approximate \eqref{LA1}. Here the function $\varphi_n\colon\R^{2n} \to \R$ is Borel-measurable and
\begin{equation}
    \begin{aligned}\label{LA5}
    \mathcal{N}_n(N,W) = [N(t_1),\ldots,N(t_n), W(t_1),\ldots, W(t_n)],
    \end{aligned}
\end{equation}
where
\begin{equation}
\label{mesh_1}
    0= t_0 < t_1 < \ldots < t_n = T 
\end{equation}
is a fixed discretization of $[0,T]$. Further, we consider the trapezoidal method $\mathcal{A}_n^T(N,W)$ based on the mesh \eqref{mesh_1}, which is defined as
\begin{equation}
\label{trap_alg_def}
    \mathcal{A}_n^T(N,W)=\sum_{i=0}^{n-1} \frac{1}{2} \big(W(t_{i+1})-W(t_i)\big) \big(N(t_{i+1})+ N(t_i)\big).
\end{equation}

\begin{theorem}
\label{lb_Levy_rea}
For the trapezoidal method \eqref{trap_alg_def} based on the equidistant mesh $t_i=iT/n$, $i\in\{0,1,\ldots,n\}$, 
it holds that
\begin{equation}
    \lim_{n\to \infty} n^{1/2}\cdot \| J(N,W) - \mathcal{A}_n^T (N,W)\|_2 
    =\lim\limits_{n\to \infty}n^{1/2}\cdot\inf\limits_{\mathcal{A}_n}\| J(N,W) - \mathcal{A}_n (N,W)\|_2=  \frac{\lambda^{1/2} T}{2}.
\end{equation}
Hence, $\mathcal{A}_n^T(N,W)$ is the optimal method among all methods of the form \eqref{LA4}.
\end{theorem}

\begin{proof}
The projection property for the conditional expectation implies for any algorithm \eqref{LA4} that 
\begin{equation}
    \begin{aligned}\label{LA6}
    \E\Big[\big| J(N,W) - \mathcal{A}_n(N,W)\big|^2\Big] 
    \geq \E\Big[\big| J(N,W) - \E\big[J(N,W)\big|\mathcal{N}_n (N,W)\big]\big|^2\Big].
    \end{aligned}
\end{equation}
This is because $\mathcal{A}_n(N,W)$ is measurable with respect to the sigma-algebra generated by $\mathcal{N}_n(N,W)$. Therefore, we also have
\begin{equation}
\label{lower_b_at}
\inf_{\mathcal{A}_n}\E\Big[\big| J(N,W) - \mathcal{A}_n(N,W)\big|^2\Big] \geq \inf\limits_{0=t_0<t_1\ldots< t_n=T}\E\Big[\big| J(N,W) - \E\big[J(N,W)\big|\mathcal{N}_n (N,W)\big]\big|^2\Big].
\end{equation}
Hence, we need to compute 
\begin{equation}
    \begin{aligned}\label{LA7}
    \E\big[ J(N,W)\big|\mathcal{N}_n (N,W)\big] = \sum_{i=0}^{n-1} \E\Bigg[ \int\limits_{t_i}^{t_{i+1}} N(t) \diff W(t) \Big|\mathcal{N}_n (N,W) \Bigg].
    \end{aligned}
\end{equation}
For all $i\in\{0,\ldots,n-1\}$ we define 
\begin{equation}
J_{t_i, t_{i+1}}(N,W) = \int\limits_{t_i}^{t_{i+1}} N(t) \diff W(t).
\end{equation}
The definition of the It\^o integral implies for all $i\in\{0,\ldots,n-1\}$ that
\begin{equation}
    \begin{aligned}\label{LA9}
    J_{t_i, t_{i+1}}(N,W) = \lim_{m\to\infty} J_m^i (N,W) \text{ in } L^2(\R).
    \end{aligned}
\end{equation}
Here 
\begin{equation}
    \begin{aligned}\label{LA10}
    J_m^i (N,W) 
    = \sum_{j=0}^{m-1} N(s_j^i) (W(s_{j+1}^i) - W(s_{j}^i)),
    \end{aligned}
\end{equation}
with $s_j^i = t_i +j(t_{i+1}-t_i)/m$ for all $j\in\{0,\ldots,m\}$. 
Further, we define $\Delta W_j^i =  W(s_{j+1}^i) - W(s_{j}^i)$
for all $i\in\{1,\ldots,n-1\}$ and $j\in\{1,\ldots,m-1\}$. Then it holds that
\begin{equation}
    \begin{aligned}\label{LA12}
    \E\big[J_m^i (N,W) \big|\mathcal{N}_n(N,W)\big] 
    = \sum_{j=0}^{m-1} \E\big[N(s_j^i) \Delta W_j^i\big| \mathcal{N}_n(N,W)\big].
    \end{aligned}
\end{equation}
Since by \cite[Lemma B.18]{AK2021} the processes $N$ and $W$ are conditionally independent given the sigma-algebra generated by $\mathcal{N}_n(N,W)$,
we obtain
\begin{equation}
    \begin{aligned}\label{LA13}
    \E\big[N(s_j^i) \Delta W_j^i\big| \mathcal{N}_n(N,W)\big]
    = \E\big[N(s_j^i) \big| \mathcal{N}_n(N)\big]\cdot\E\big[\Delta W_j^i\big| \mathcal{N}_n(W)\big].
    \end{aligned}
\end{equation}
Using \cite[Lemma 8]{Hertling2001} and \cite[Lemma 3.1]{PP2016}, we obtain for all $s \in[t_i, t_{i+1}]$ 
\begin{equation}
    \begin{aligned}\label{LA14}
    & \E\big[N(s) \big| \mathcal{N}_n(N)\big]
    = \frac{N(t_{i+1}) (s -t_i) + N(t_i)(t_{i+1}-s)}{t_{i+1}-t_i}
    \end{aligned}
\end{equation}
and
\begin{equation}
    \label{LA15}
     \E\big[\Delta W_j^i\big| \mathcal{N}_n(W)\big]
    = \frac{\big(W(t_{i+1})-W(t_i)\big) (s_{j+1}^i -s_{j}^i)}{t_{i+1}-t_i}.
\end{equation}
Plugging \eqref{LA13},\eqref{LA14}, and \eqref{LA15} into \eqref{LA12}, we obtain
\begin{equation}
    \begin{aligned}\label{LA16}
    &\E\big[J_m^i (N,W) \big|\mathcal{N}_n(N,W)\big] 
    = \sum_{j=0}^{m-1} \E\big[N(s_j^i) \big| \mathcal{N}_n(N)\big]\cdot\E\big[\Delta W_j^i\big| \mathcal{N}_n(W)\big]\\
        &= \frac{W(t_{i+1})-W(t_i)}{t_{i+1}-t_i} \sum_{j=0}^{m-1} \E\big[N(s_j^i) \big| \mathcal{N}_n(N)\big]\cdot
        (s_{j+1}^i -s_{j}^i).
    \end{aligned}
\end{equation}
Note that $\displaystyle{\sum_{j=0}^{m-1} \E\big[N(s_j^i) \big| \mathcal{N}_n(N)\big] (s_{j+1}^i -s_{j}^i)}$ is a (pathwise) Riemann approximation of the stochastic process $\big(\E\big[N(s) \big| \mathcal{N}_n(N)\big]\big)_{s\in [t_i,t_{i+1}]}$ and has continuous sample paths. Hence, it holds for all $i\in\{0,\ldots,m-1\}$ that
\begin{equation}
    \begin{aligned}\label{LA17}
    \lim_{m\to\infty} \E\big[J_m^i (N,W) \big|\mathcal{N}_n(N,W)\big] 
    = \frac{W(t_{i+1})-W(t_i)}{t_{i+1}-t_i} \int\limits_{t_i}^{t_{i+1}} \E\big[N(t)\big|\mathcal{N}_n(N)\big] \diff t \text{ a.s.}
    \end{aligned}
\end{equation}
In addition, \eqref{LA10} and Jensen's inequality for the conditional expectation imply
\begin{equation}
    \begin{aligned}\label{LA18}
    &\E\Big[\big| \E[J_{t_i, t_{i+1}}(N,W) |\mathcal{N}_n(N,W)] - \E[J_{m}^i(N,W) |\mathcal{N}_n(N,W)]\big|^2\Big]\\
    &\leq \E\Big[\big| J_{t_i, t_{i+1}}(N,W) - J_{m}^i(N,W) \big|^2\Big]
    \to 0 \text{ as } m\to\infty.
    \end{aligned}
\end{equation}
Hence, by \eqref{LA18} holds
\begin{equation}
    \begin{aligned}\label{LA19}
    &\E\big[J_{m}^i(N,W) \big|\mathcal{N}_n(N,W)\big] \to \E\big[J_{t_i, t_{i+1}}(N,W) \big|\mathcal{N}_n(N,W)\big] \text{ as } m\to\infty \text{ in } L^2(\Omega),
    \end{aligned}
\end{equation}
and by \eqref{LA17} holds
\begin{equation}
    \begin{aligned}\label{LA20}
    &\E\big[J_{m}^i(N,W) \big|\mathcal{N}_n(N,W)\big] \to \frac{W(t_{i+1})-W(t_i)}{t_{i+1}-t_i} \int\limits_{t_i}^{t_{i+1}} \E\big[N(t)\big|\mathcal{N}_n(N)\big] \diff t \text{ as } m\to\infty \text{ a.s.}
    \end{aligned}
\end{equation}
Convergence in $L^2(\Omega)$ as well as almost sure convergence imply convergence in probability. Additionally, by uniqueness of the limit in probability, we get that for all $i\in\{0,\ldots,m-1\}$,
\begin{equation}
    \begin{aligned}\label{LA21}
    &\E\big[J_{t_i, t_{i+1}}(N,W) \big|\mathcal{N}_n(N,W)\big] = \frac{W(t_{i+1})-W(t_i)}{t_{i+1}-t_i} \int\limits_{t_i}^{t_{i+1}} \E\big[N(t)\big|\mathcal{N}_n(N)\big] \diff t \text{ a.s.}
    \end{aligned}
\end{equation}
Further, it holds that
\begin{equation}
    \label{LA22}
    \int\limits_{t_i}^{t_{i+1}} \E\big[N(t)\big|\mathcal{N}_n(N)\big] \diff t 
    = \frac{1}{2} \big(N(t_{i+1})+ N(t_i)\big)(t_{i+1} -t_i).
\end{equation}
When we plug \eqref{LA22} in \eqref{LA21}, we get for all $i\in\{0,\ldots,m-1\}$ that
\begin{equation}
    \begin{aligned}\label{LA23}
    \E\big[J_{t_i, t_{i+1}}(N,W) \big|\mathcal{N}_n(N,W)\big] 
    = \frac{1}{2} (W(t_{i+1})-W(t_i)) (N(t_{i+1})+ N(t_i)).
    \end{aligned}
\end{equation}
Combining \eqref{LA23} and \eqref{LA7} yields $\E[ J(N,W)|\mathcal{N}_n (N,W)] = \mathcal{A}^T_n(N,W)$, i.e.~$\E[ J(N,W)|\mathcal{N}_n (N,W)]$ corresponds to the trapezoidal method.

Next, we calculate the error of the trapezoidal method to get the minimal possible error among all methods of the form \eqref{LA4}. 
Consider the step process given for all $t\in[0,T]$ by
\begin{equation}
    \begin{aligned}\label{LA25}
    \hat N_n(t) := \sum_{i=0}^{n-1} \mathds{1}_{(t_i, t_{i+1}]}(t) \frac{N(t_i) + N(t_{i+1})}{2}.
    \end{aligned}
\end{equation}
Note that the process $(\hat N_n(t))_{t\in[0,T]}$ is not adapted to the filtration $(\mathcal{F}_t)_{t\in[0,T]}$. However, it is adapted to the sigma-algebra generated by $ \mathcal{F}_t^W$ and $\mathcal{F}_T^N$, called $\mathcal{\widetilde F}_t$, for all $t\in[0,T]$. Additionally,  $(W(t))_{t\in[0,T]}$ is still a scalar Wiener process with respect to the filtration $(\mathcal{\widetilde F}_t)_{t\in[0,T]}$, since $\mathcal{F}_T^N$ and $\mathcal{F}_T^W$ are independent. Hence,   
\begin{equation}
J(\hat N_n, W) = \int\limits_0^T \hat N_n (t) \diff W(t)
\end{equation} 
is a well-defined  It\^o integral of the $(\mathcal{\widetilde F}_t)_{t\in[0,T]}$-simple process $(\hat N_n(t))_{t\in[0,T]}$ and it holds
\begin{equation}
    \begin{aligned}\label{LA27}
    J(\hat N_n, W) = \sum_{i=0}^{n-1}  \frac{N(t_i) + N(t_{i+1})}{2} (W(t_{i+1}) -W(t_i))=\mathcal{A}_n^T (N,W).
    \end{aligned}
\end{equation}
Using that $(N(t)-\hat N_n(t))_{t\in[0,T]}$ is a $(\mathcal{\widetilde F}_t)_{t\in[0,T]}$-progressively measurable  process, the It\^o isometry 
and Jensen's inequality we obtain
\begin{equation}
    \begin{aligned}\label{LA29}
    &\E\Big[ \big| J(N,W) - \mathcal{A}_n^T (N,W)\big|^2\Big]
    = \sum_{i=0}^{n-1} \int\limits_{t_i}^{t_{i+1}} \E\Big[ \big|N(t) - \hat N_n(t)\big|^2 \Big] \diff t\\
    &= \frac{1}{4} \sum_{i=0}^{n-1} \int\limits_{t_i}^{t_{i+1}} \Big(\E\Big[\big(N(t) - N(t_i)\big)^2 \Big] - 2\E\Big[ N(t) - N(t_i) \Big]\cdot\E\Big[ N(t_{i+1})-N(t)\Big]\\
    &\quad\quad\quad\quad + \E\Big[ \big(N(t_{i+1}) -N(t)\big)^2 \Big]\Big) \diff t\\
    &= \frac{\lambda}{4} \sum_{i=0}^{n-1} (t_{i+1}-t_i)^2 + \frac{\lambda^2}{12} \sum_{i=0}^{n-1} (t_{i+1}-t_i)^3\geq \frac{\lambda}{4n} \Big(\sum_{i=0}^{n-1} (t_{i+1}-t_i)\Big)^2 + \frac{1}{n^2}\,\frac{\lambda^2}{12} \Big(\sum_{i=0}^{n-1} (t_{i+1}-t_i)\Big)^3\\
    &\geq \frac{\lambda T^2}{4n} + \frac{\lambda^2 T^3}{12 n^2}.
    %\geq \frac{\lambda T^2}{4n},
    \end{aligned}
\end{equation}
This implies
\begin{equation}
    \inf\limits_{0=t_0<t_1<\ldots t_n=T}\E\Big[ \big| J(N,W) - \mathcal{A}_n^T (N,W)\big|^2\Big]\geq \frac{\lambda T^2}{4n} + \frac{\lambda^2 T^3}{12 n^2}.
\end{equation}
Hence, using \eqref{lower_b_at} we conclude that
\begin{equation}
    \begin{aligned}
    \label{LA31}
    n^{1/2}\cdot\inf_{\mathcal{A}_n}  \| J(N,W) - \mathcal{A}_n(N,W) \|_{L^2(\Omega)} \geq \sqrt{\frac{\lambda T^2}{4}+\frac{\lambda^2 T^3}{12n}}.
    \end{aligned}
\end{equation}
For the trapezoidal method $\mathcal{A}^T_n(N,W)$ based on the equidistant mesh $t_i = iT/n$, $i\in\{0,1,\ldots,n\}$ it follows from \eqref{LA29} that 
\begin{equation}
    \label{LA32}
    \E\Big[ \big| J(N,W) - \mathcal{A}_n^T (N,W)\big|^2\Big] 
    = \frac{\lambda T^2}{4n}  + \frac{\lambda^2 T^3}{12n^2},
\end{equation}
and hence 
\begin{equation}
\label{low_b_1}
    n^{1/2}\cdot\inf\limits_{\mathcal{A}_n} \| J(N,W) - \mathcal{A}_n (N,W)\|_2\leq n^{1/2}\cdot \| J(N,W) - \mathcal{A}^T_n (N,W)\|_2= \sqrt{\frac{\lambda T^2}{4}+\frac{\lambda^2 T^3}{12n}}.
\end{equation}
Combining \eqref{LA31} and \eqref{low_b_1} we obtain
\begin{equation}
    \begin{aligned}\label{LA33}
    & n^{1/2}\cdot\inf\limits_{\mathcal{A}_n} \| J(N,W) - \mathcal{A}_n (N,W)\|_2 
    = \sqrt{\frac{\lambda T^2}{4}+\frac{\lambda^2 T^3}{12n}}.
    \end{aligned}
\end{equation}
This and \eqref{LA32} prove the claim.
\end{proof}

\begin{remark}\label{RemMultSDE}
Consider any class of coefficients of multidimensional SDEs for which \eqref{2dim_sde} is a subproblem. Then by Theorem \ref{lb_Levy_rea}, in the worst case setting with respect to the coefficients, the error cannot be smaller than $\Omega(n^{-1/2})$. Therefore, no matter if the JCC \eqref{JCC} is satisfied or not, we can apply the Euler--Maruyama (or randomized Euler--Maruyama) scheme in order to achieve the optimal $L^2(\Omega)$-error bound $O(n^{-1/2})$, for example, \cite{PMPP2014,PMPP2017}.
This is in contrast to the scalar case with JCC, where the randomized Milstein scheme outperforms the Euler scheme.
\end{remark}

\begin{remark}
In Theorems \ref{opt_rm_scalar} and \ref{lb_Levy_rea} we have considered only the $L^2$-error. Matching upper and lower bounds that depend on $p$ remain an open problem. Our numerical experiments in Section \ref{num} suggest that for jump-diffusion SDEs the error indeed depends on $p$.
\end{remark}

\section{Numerical experiments}
\label{num}

We implement\footnote{The program code is available as ancillary file from the arXiv page of this paper (arXiv:2212.00411).} the randomized Milstein algorithm for the SDE 
\begin{equation}
    \begin{aligned}\label{ExSDE}
    \left\{
    \begin{array}{lll}
    \diff X(t) = \sin(M\cdot X(t) (1+t)^{\varrho_1}) \diff t + \cos(M\cdot X(t) \cdot (1+t)^{\varrho_2}) \diff W(t)& \, \\
    \qquad\qquad\qquad\qquad \qquad\qquad\qquad\qquad \qquad + \Big(-X(t) +\frac{\pi}{2M\cdot (1+t)^{\varrho_2}}\Big) \diff N(t),  & t\in[0,1], \\
    X(0) = 1. & \, \\
    \end{array}
    \right.
    \end{aligned}
\end{equation}
This SDE has already been implemented in \cite{Morkisz2020} in the jump-free case. We choose the jump coefficient such that the JCC is satisfied.
The verification of Assumption \ref{Ass} is straight forward, for the JCC \eqref{JCC}, see Remark \ref{class_jcc}.
In our simulations we set $\lambda = 100$, $M =100$,  $\varrho_1 =0.1$, and $\varrho_2 = 0.6$. 

We estimate the $L^p$-error similar as in \cite[p.~14]{PS20} by
\begin{equation}
    \begin{aligned}
    \operatorname{error}(k) = \operatorname{mean}\big( \big| X^{(k)}(T) - X^{(k-1)}(T)\big|^p \big)^{\frac{1}{p}}.
    \end{aligned}
\end{equation}
Here, $X^{(k)}(T)$ is the approximation of $X(T)$ with step size $\delta^{(k)}$, where $\delta^{(k)} = 2^{-k}$ for $k\in\N$. The $\operatorname{mean}$ is taken over $2^{16}$ sample paths.

\begin{remark}
The interesting part of the implementation is the simulation of the randomization, i.e.~the values $\xi_i$ are computed. We proceed by  
first simulating independent uniformly distributed random variables $\xi_i$ on the corresponding intervals for the finest discretization grid. Then we iteratively compute the values for the discretization grid with doubled step size as follows: One time interval in the larger grid consists of two time intervals of equal length in the finer grid. For those two intervals we have simulated two values $\xi_i$. Now we simulate an independent Bernoulli$(0.5)$ random variable that determines which of the values $\xi_i$ we take. This choice is then uniformly distributed on the interval of the large grid and hence consistent with the randomized Milstein algorithm.
\end{remark}

For $p\in [2,\infty)$ we obtain by Theorem \ref{ThmUpperBound} the theoretical convergence rate \[\displaystyle{\min\Big\{\frac{2}{p}, \varrho_1 + \frac{1}{p}, \varrho_2, \varrho_2\Big\} = \min\Big\{\frac{2}{p}, 0.1 + \frac{1}{p}, 0.6\Big\}}.\] For $p =1$ we take as theoretical convergence rate the same rate as for $p=2$, because the $L^1$-error can be estimated by the $L^2$-error using the Cauchy-Schwarz inequality.
In Figure \ref{fig:1} we plot the $\log_2(\operatorname{error}(k))$ over $\log_2(\delta^{(k)}))$ for $p\in\{1,2,3,4\}$ and the corresponding theoretical convergence orders.

\begin{figure}[h]
 \centering
 \includegraphics[width=16cm]{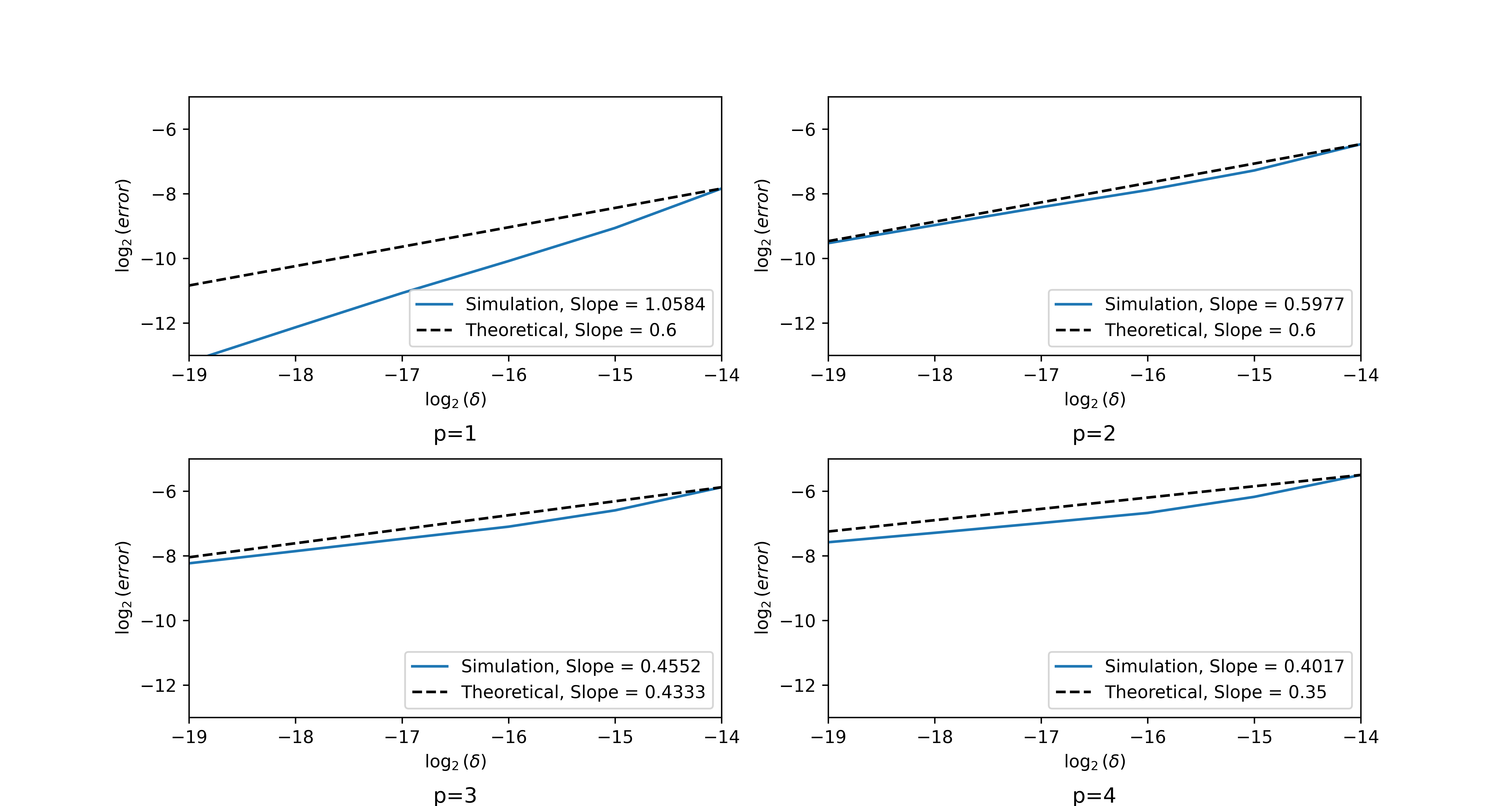}
 \caption{Error estimates and theoretical convergence order for $p\in\{1,2,3,4\}$}
 \label{fig:1}
\end{figure}

We see that the observed convergence order is decreasing with increasing $p$. Further we notice that for $p =1$ the convergence of the simulation is higher than the theoretical convergence rate. This is reasonable because we took the rate of the $L^2$-error.
For $p=2$ we observe that the simulation confirms the theoretical results; the slope of the simulation matches the convergence rate, which we proved to be optimal. Also for $p=3$ and $p=4$ the simulations confirm the theoretical results, since the simulation converges at least as fast as the theoretically obtained upper bound; we have not proven any lower bound.

Next, we regress the slope of the simulated $\log_2(\operatorname{error}(k))$ in dependence of the corresponding $\log_2(\delta^{(k)}))$ for all $p\in\{1,\ldots,8\}$ and compare it to the theoretical upper bounds on the convergence rates we have proven, see Figure \ref{fig:2}.
We observe that for the simulations the convergence order is dependent on $p$, which confirms also this theoretical finding.

\begin{figure}[h]
 \centering
 \includegraphics[width=12cm,height=6cm,keepaspectratio]{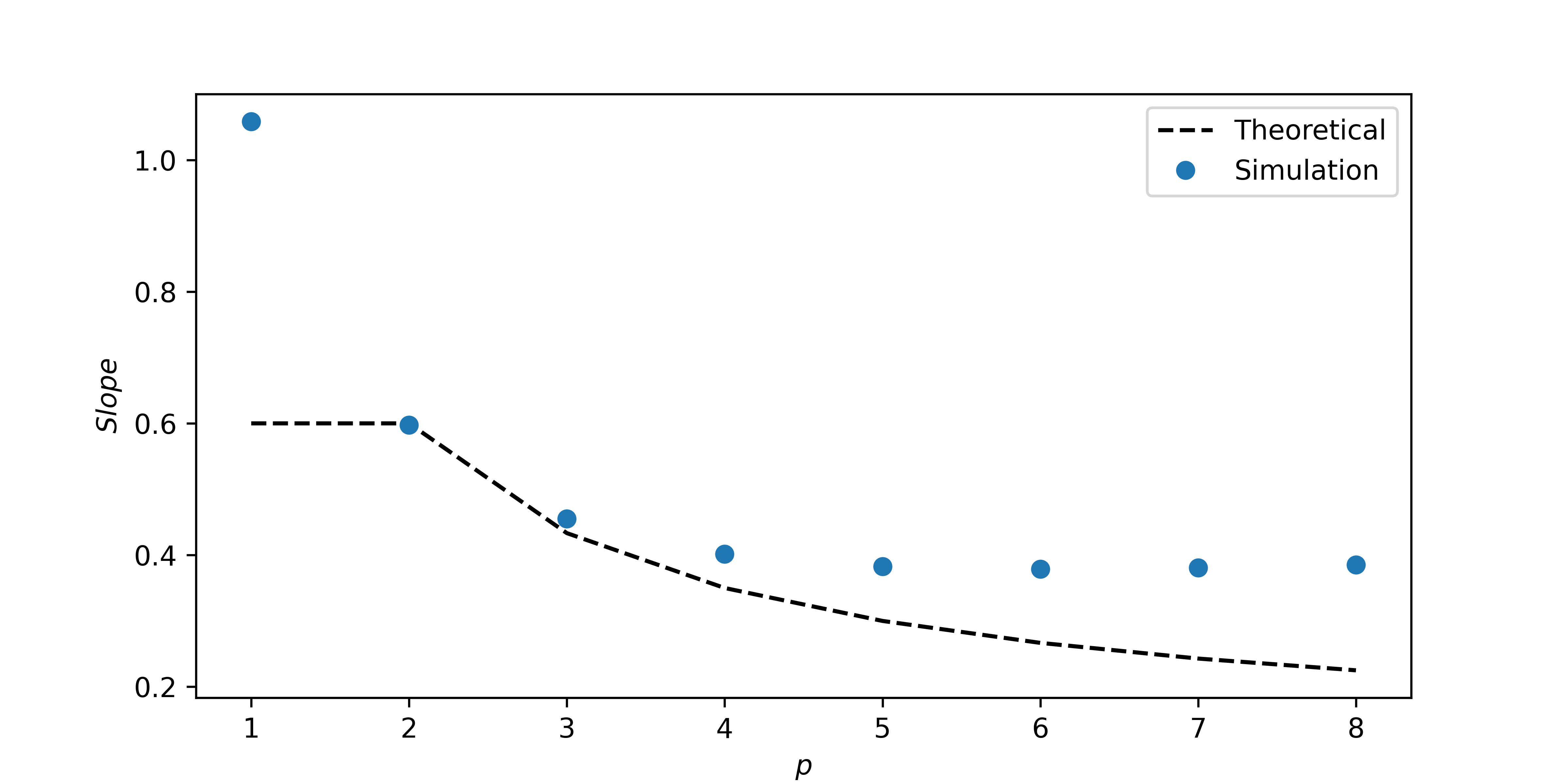}
 \caption{Slopes of the simulation (estimated by linear regression) in comparison to theoretical convergence rates}
 \label{fig:2}
\end{figure}

\begin{remark}\rm
\label{linear_jcc}
For the very simple example of SDEs with linear coefficients, we did not observe an $L^p$ dependence of the error.
\end{remark}

%%%%%%%%%%%%%%
\begin{remark}\rm
\label{class_jcc}
Let us assume that the diffusion coefficient is of the form $\sigma(t,y)=F(\alpha(t) y+\beta(t))$ while the jump coefficient $\rho(t,y)=-y+\gamma(t)$ for some functions $F:\mathbb{R}\to\mathbb{R}$ and $\alpha,\beta,\gamma:[0,T]\to\mathbb{R}$. Moreover, let us assume that there exists $x_0\in\mathbb{R}$ such that 
\begin{itemize}
    \item $F(x_0)=0$,
    \item $\alpha(t)\cdot\gamma(t)+\beta(t)=x_0$ for all $t\in [0,T]$.
\end{itemize}
Then the JCC \eqref{JCC} is satisfied for the pair $(\sigma,\rho)$. This provides a new class of functions $(\sigma,\rho)$ satisfying the JCC which may, in contrast to the class considered in
\cite{PBL2010}, be nonlinear.
\end{remark}

%%%%%%%%%%%%%%

\appendix

\section{Appendix}
The proof of the following lemma is straightforward and will be omitted.
\begin{lemma}
\label{MEstL}
Under Assumption \ref{Ass} there exists a constant $K_7\in(0,\infty)$ such that for $f\in\{\mu,\sigma,\rho\}$ and for all $t_1,t_2,t,u\in [0,T]$,
\begin{equation}
\begin{aligned}\label{MEstL1}
| \alpha_1(f,t,u)| &\leq K_7 (1+|X(u)|),\\
| \beta(f,t,u)| &\leq K_7 (1+|X(u)|),\\
|\beta(\mu,t_1,u)-\beta(\mu,t_2,u)|&\leq K_7 (1+|X(u)|^2)\cdot |t_1-t_2|^{\varrho_1},\\
| \gamma(f,t,u)| &\leq K_7 (1+|X(u-)|),\\
|\gamma(\mu,t_1,u)-\gamma(\mu,t_2,u)|&\leq K_7(1+|X(u-)|)|t_1-t_2|^{\varrho_1}.
\end{aligned}
\end{equation}
\end{lemma}

%%%%%%%%%%%

The following estimate is a direct consequence of the H\"older, the Burkholder-Davis-Gundy, and the Kunita inequalitiy, see \cite{kunita2004}.
\begin{lemma}\label{BDGaKunita}
Let $q\in [2,\infty)$, $a,b\in[0,T]$ with $a<b$, $Z\in\{\operatorname{Id},W,N\}$, $Y = (Y(t))_{t\in[a,b]}$ is a predictable stochastic process such that 
\begin{equation}
\E\Big[\int\limits_a^b |Y(t)|^q \diff t\Big] < \infty
\end{equation}
Then there exists a constant $\hat c\in(0,\infty)$ such that for all $t\in[a,b]$ it holds that
\begin{equation}
\begin{aligned}
\E\Biggl[ \sup_{s\in[a,t]} \Biggl|\int\limits_a^s Y(u) \diff Z(u)\Biggl|^q \Biggr] \leq \hat c \int\limits_a^t \E[|Y(u)|^q]\diff u. 
\end{aligned}
\end{equation}
\end{lemma}
%%%%%%%%%%%%%%
\section*{Acknowledgements}

V. Schwarz and M. Sz\"olgyenyi are supported by the Austrian Science Fund (FWF): DOC 78.

%%%%%%%%%%%%%%%%%%%%%%%%%%%%%%%%%%%%%%%%%%%%%%%%%%%%%%%%%%%%%%%%%%%%%%%%%%%%%%%%%%%%%%%%%%%%%%%%%%%%%%%%%%%%%%%

%%%%%%%%%%%%%%%%%%%%%%%%%%%%%%%%%%%%%%%%%%%%%%%%%%%%%%%%%%%%%%%%%%%%%%%%%%%%%%%%%%%%%%%%%%%%%%%%%%%%%%%%%%%%%%%

\vspace{2em}
\centerline{\underline{\hspace*{16cm}}}

\noindent Pawe{\l} Przyby{\l}owicz \\
Faculty of Applied Mathematics, AGH University of Krakow, Al.~Mickiewicza 30, 30-059 Krakow, Poland\\
pprzybyl@agh.edu.pl\\

\noindent Verena Schwarz \Letter \\
Department of Statistics, University of Klagenfurt, Universit\"atsstra\ss{}e 65-67, 9020 Klagenfurt, Austria\\
verena.schwarz@aau.at\\

\noindent Michaela Sz\"olgyenyi \\
Department of Statistics, University of Klagenfurt, Universit\"atsstra\ss{}e 65-67, 9020 Klagenfurt, Austria\\
michaela.szoelgyenyi@aau.at\\

%%%%%%%%%%%%%%%%%%%%%%%%%%%%%%%%%%%%%%%%%%%%%%%%%%%%%%%%%%%%%%%%%%%%%%%%%%%%%%%%%%%%%%%%%%%%%%%%%%%%%%%%%%%%%%%

\end{document}